\documentclass{amsart}

\usepackage{amsfonts}
\usepackage{amsmath}
\usepackage{amsthm}
\usepackage{graphicx}
\usepackage{wasysym}
\usepackage{amssymb}
\usepackage{mathrsfs}
\usepackage{amscd}
\usepackage{lscape}
\usepackage{longtable}
\usepackage{rotating}
\usepackage{phonetic}
\usepackage{pstricks}
\usepackage{pst-node}
\usepackage[crop=off]{auto-pst-pdf}
\usepackage{etex}
\usepackage{float}

\input prepictex
\input pictex
\input postpictex

\newtheorem{lemma}{Lemma}[section]
\newtheorem{thm}[lemma]{Theorem}

\newtheorem{prop}[lemma]{Proposition}
\newtheorem{corol}[lemma]{Corollary}
\newtheorem{conj}[lemma]{Conjecture}

\newtheorem{example}[lemma]{Example}

\title{INEQUALITIES BETWEEN GAMMA-POLYNOMIALS OF GRAPH-ASSOCIAHEDRA}
\author{NATALIE AISBETT}
\address{School of Mathematics and Statistics\\
University of Sydney, NSW, 2006\\
Australia}
\email{N.Aisbett@maths.usyd.edu.au}
\date{}

\begin{document}
\maketitle{}
\numberwithin{figure}{section}
\numberwithin{equation}{section}

\begin{abstract}
We prove a conjecture of Postnikov, Reiner and Williams by defining a partial order on the set of tree graphs with $n$ vertices that induces inequalities between the $\gamma$-polynomials of their associated graph-associahedra. The partial order is given by relating trees that can be obtained from one another by operations called tree shifts. We also show that tree shifts lower the $\gamma$-polynomials of graphs that are not trees, as do the flossing moves of Babson and Reiner.
\end{abstract}

\begin{section}{Introduction}
For any building set $\mathcal{B}$ there is an associated simple polytope $P_\mathcal{B}$ called the \emph{nestohedron} (see \cite{po} Section 7 and \cite{prw} Section 6). When $\mathcal{B} = \mathcal{B}(G)$ is the building set determined by a graph $G$, $P_{\mathcal{B}(G)}$ is the well-known graph-associahedron of $G$ (see \cite{bv}, \cite{Er}, \cite{prw} Sections 7 and 12, and \cite{vol}). The numbers of faces of $P_\mathcal{B}$ of each dimension are conveniently encapsulated in its $\gamma$-polynomial $\gamma(\mathcal{B}) =\gamma(P_\mathcal{B})$ (see \cite{prw} Section 1 for the definition). Postnikov, Reiner and Williams conjectured the following monotonicity property of the $\gamma$-polynomials of the graph-associahedra of trees.
\begin{conj}
\cite[Conjecture 14.1]{prw}. There exists a partial order $\le$ on the set of (unlabelled, isomorphism classes of) trees with $n$ vertices, with the following properties:
\begin{itemize}
\item $\mathrm{Path_n}$ is the unique $\le$-minimal element,
\item $K_{1,n-1}$ is the unique $\le$-maximal element,
\item $T \le T'$ implies $\gamma(\mathcal{B}(T)) \le \gamma(\mathcal{B}(T'))$.
\end{itemize}
\label{bigconj}
\end{conj}
\noindent Here $\mathrm{Path_n}$ denotes the graph that is a path with $n$ vertices, and $K_{1,n-1}$ is the graph with $n$ vertices with exactly one vertex of degree $n-1$ and $n-1$ vertices of degree 1.\\

This conjecture implies the following lower and upper bounds for the $\gamma$-polynomial of a tree $T$ with $n$ vertices \begin{equation}\gamma(\mathcal{B}(\mathrm{Path_n})) \le \gamma(\mathcal{B}(T)) \le \gamma(\mathcal{B}(K_{1,n-1})).\label{eqone}\end{equation}
These upper and lower bound theorems have been proven by Buchstaber and Volodin \cite[Theorem 9.4]{bv}. Moreover, they show that the lower bound is attained only for $\mathrm{Path_n}$ and the upper bound is attained only for $K_{1,n-1}.$ Their proof relies on some general results about $\gamma$-polynomials of flag nestohedra which were announced in \cite{vol} and whose proofs are included in \cite{bv}; see Lemmas \ref{equivlem}, \ref{vollem}, \ref{combeq} and theorems \ref{flagbuild}, \ref{posthm} and \ref{volthm}. Note that the methods of Buchstaber and Volodin require one to work with the more general class of flag nestohedra in order to deduce the results about graph-associahedra. In this paper we make use of these theorems to show that Conjecture \ref{bigconj} can be proven with the relation of tree shifts that we define.\\

We also use these theorems to show that flossing moves lower the $\gamma$-polynomial. Flossing moves were originally defined in \cite{br} Section 4.2 and it was suggested in \cite{prw} Section 14 that they might lower the $\gamma$-polynomial. Our definition of flossing move is more general than that in \cite{br} as it can be applied to any pair of leaves that floss a vertex, and it does not have to be applied to a tree graph. \\

Section 2 contains preliminary definitions and results relating to polytopes and building sets. Section 3 contains more specific results relating to the $\gamma$-polynomial that are needed for the main theorems in Sections 4 and 5. Section 4 introduces tree shifts and in Theorem \ref{bigthm} we show that they lower the $\gamma$-polynomial of the associated graph-associahedra. We then prove Conjecture \ref{bigconj}, in Theorem \ref{partial}. Section 5 introduces flossing moves and Theorem \ref{bigthmtwo} shows that they lower the $\gamma$-polynomials.\\

\textbf{Acknowledgements}\\

This paper forms part of my PhD research in the School of Mathematics and Statistics at the University of Sydney. I would like to thank my supervisor Anthony Henderson for his assistance and suggested improvements. I would also like to thank the referee for their careful reading and helpful comments.

\end{section}

\begin{section}{Building sets and nestohedra}

A \emph{building set} $\mathcal{B}$ on a finite set $S$ is a set of non empty subsets of $S$ such that
\begin{itemize}
\item For any $I,~ J \in \mathcal{B}$ such that $I \cap J \ne \emptyset$, $I \cup J \in \mathcal{B}$.
\item $\mathcal{B}$ contains the singletons $\{i\}$, for all $i \in S$.
\end{itemize}
$\mathcal{B}$ is \emph{connected} if it contains $S$. For any building set $\mathcal{B}$, $\mathcal{B}_{max}$ denotes the set of maximal elements of $\mathcal{B}$ with respect to inclusion. The elements of $\mathcal{B}_{max}$ form a disjoint union of $S$, and if $\mathcal{B}$ is connected then $\mathcal{B}_{max} = \{S\}$. Building sets $\mathcal{B}_1$, $\mathcal{B}_2$ on $S$ are \emph{equivalent}, denoted $\mathcal{B}_1 \cong \mathcal{B}_2$, if there is a permutation $\sigma: S \rightarrow S$ that induces a one to one correspondence $\mathcal{B}_1 \rightarrow \mathcal{B}_2$. \\

\begin{example}Let $G$ be a graph with no loops or multiple edges, with $n$ vertices labelled distinctly from $[n]$. Then the graphical building set $\mathcal{B}(G)$ is the set of subsets of $[n]$ such that the induced subgraph of $G$ is connected. $\mathcal{B}(G)_{max}$ is the set of connected components of $G$.
\end{example}

Let $\mathcal{B}$ be a building set on $S$ and $I \subseteq S$. The \emph{restriction of $\mathcal{B}$ to $I$} is the building set $$\mathcal{B}|_{I}: = \{J ~|~ J \subseteq I, ~\hbox{and }~J \in \mathcal{B} \}~~\hbox{on $I$}.$$ The \emph{contraction of $\mathcal{B}$ by $I$} is the building set
$$\mathcal{B}/I: = \{J-(J \cap I)~|~J \in \mathcal{B},~J \not \subseteq I\}~~\hbox{on $S-I$.}$$

\begin{example} If $G$ is a graph on $[n]$, and $I \in \mathcal{B}(G)$, then $\mathcal{B}(G)/I =\mathcal{B}(G')$ where $G'$ is the graph on $[n] -I$ such that any two vertices $i,~j \in [n] -I$ are adjacent if they are adjacent in $G$, or both $i$ and $j$ are adjacent to vertices in $I$ in the full graph $G$.
\end{example}

Given a building set $\mathcal{B}$, a subset $N \subseteq \mathcal{B} \backslash \mathcal{B}_{max}$ is a \emph{nested set} if it satisfies
\begin{itemize}
\item For any $I,~ J \in N$, either $I \subseteq J$, $J \subseteq I$, or $I \cap J = \emptyset$.
\item For any collection of $k \ge 2$ disjoint subsets $J_1,....,J_k \in N$, the union $J_1 \cup \cdots \cup J_k \not \in \mathcal{B}$.
\end{itemize}

\noindent The \emph{nested set complex} $\Delta_\mathcal{B}$ is the simplicial complex on $\mathcal{B} - \mathcal{B}_{max}$ whose faces are the nested sets. We associate a polytope to a building set as follows. Let $e_1,....,e_n$ denote the endpoints of the coordinate vectors in $\mathbb{R}^n$. Given $I \subseteq [n]$, define the simplex $\Delta_I : = ConvexHull(e_i~|~i \in I)$. Let $\mathcal{B}$ be a building set on $[n]$. The \emph{nestohedron} $P_\mathcal{B}$ is a polytope given by the Minkowski sum of the simplices $\Delta_I$ for all $I \in B$

$$P_{\mathcal{B}} := \sum_{I \in \mathcal{B}}\Delta_I.$$

\noindent If $\mathcal{B}$ is a graphical building set $P_{\mathcal{B}}$ is known as the \emph{graph-associahedron}. The nestohedron is related to the nested sets of any building set $\mathcal{B}$, as described in the following theorem.

\begin{thm}
\cite[Theorem 7.4]{po} \cite[Theorem 3.14]{fs}. Let $\mathcal{B}$ be a building set on $[n]$. The nestohedron $P_{\mathcal{B}}$ is a simple polytope of dimension $n - |\mathcal{B}_{max}|$. The simplicial polytope polar dual to $P_\mathcal{B}$ has boundary complex isomorphic to $\Delta_{\mathcal{B}}$. \label{facethm}
\end{thm}

For a simple $d$ dimensional polytope $P$, the $f$-polynomial, $h$-polynomial and $\gamma$-polynomial are polynomials in $\mathbb{Z}[t]$ defined as follows:

$$f(P)(t):= f_0 + f_1t + \cdots +f_{d}t^d,$$ where $f_i$ is the number of $i$-dimensional faces of $P$. The $h$-polynomial is given by

$$h(P)(t+1): = f(P)(t),$$ and it is known to be positive and symmetric. Since it is symmetric, it can be written

$$\sum_{i=0}^dh_it^i = \sum_{i=0}^{\lfloor\frac{d}{2}\rfloor}\gamma_it^i(1+t)^{d-2i},$$ for some $\gamma_i \in \mathbb{Z}$, and the $\gamma$-polynomial is given by
$$\gamma(P)(t): = \gamma_0 + \gamma_1t + \cdots +\gamma_{\lfloor \frac{d}{2}\rfloor}t^{\lfloor \frac{d}{2}\rfloor }.$$

\noindent If a polytope $P$ is combinatorially equivalent to $P_1 \times P_2 \times \cdots \times P_n$ where $P_1,...,P_n$ are a set of polytopes, then by the definition of combinatorial equivalence we have that $f(P) =  f(P_1)f(P_2)...f(P_n)$, and consequently $\gamma(P) = \gamma(P_1)\gamma(P_2)...\gamma(P_n)$. When $\mathcal{B}$ is a building set, we denote the $\gamma$-polynomial for $P_{\mathcal{B}}$ by $\gamma(\mathcal{B})$. If $\mathcal{B}$ and $\mathcal{B'}$ are building sets, the notation $\gamma(\mathcal{B}) \le \gamma(\mathcal{B'})$ implies that for all $i$ the coefficient of $t^i$ in $\gamma(\mathcal{B})$ is less than or equal to the coefficient of $t^i$ in $\gamma(\mathcal{B'})$. \\

A $d-1$-dimensional face of a $d$-dimensional polytope is called a \emph{facet}. A simple polytope $P$ is \emph{flag} if any collection of pairwise intersecting facets has non empty intersection. A building set $\mathcal{B}$ is \emph{flag} if $P_{\mathcal{B}}$ is flag. The conditions in Proposition \ref{flagprop} determine whether a building set is flag.

\begin{prop}
\cite[Proposition 7.1]{prw}.
For a building set $\mathcal{B}$, the following are equivalent:
\begin{itemize}
\item[(1)] $P_{\mathcal{B}}$ is flag.\\
\item[(2)] If $J_1,....,J_m$, $m \ge 2$, are disjoint and $J_1 \cup \cdots \cup J_m \in \mathcal{B}$, then the sets can be reindexed so that for some  $k$ such that $1 \le k \le m-1$, $J_1 \cup \cdots \cup J_k \in \mathcal{B}$ and $J_{k+1} \cup \cdots \cup J_m \in \mathcal{B}$.\\
\item[(3)] If $N \subseteq \mathcal{B} \backslash \mathcal{B}_{max}$ such that
\begin{itemize}
\item for any $I,~J \in N$ either $I \subseteq J$, $J \subseteq I$ or $I \cap J = \emptyset$,
\item for any $I,~J \in N$ such that $I \cap J = \emptyset$, one has $I \cup J \not \in \mathcal{B}$,
\end{itemize}
then $N$ is a nested set.\label{flagprop}
\end{itemize}
\end{prop}

It follows from Proposition \ref{flagprop} that a graphical building set is flag. A \emph{minimal flag building set} $\mathcal{D}$ on a set $S$ is a connected building set on $S$ that is flag, such that that no proper subset of its elements form a connected flag building set on $S$. Minimal flag building sets are described in detail in \cite[Section 7.2]{prw}. They take the form of a binary tree, where the vertices biject to elements of $\mathcal{D}$, and the direct descendants of any non leaf vertex that represents an element $I \in \mathcal{D}$ are the two elements in $\mathcal{D}$ whose disjoint union is $I$. For any minimal flag building set $\mathcal{D}$, $\gamma(\mathcal{D}) =1$ (see \cite{prw} Section 7.2).\\

Let $\mathcal{B}$ be a building set. A \emph{binary decomposition} or \emph{decomposition} of a non singleton element $I \in \mathcal{B}$ is a set $\mathcal{D} \subseteq \mathcal{B}$ that forms a minimal flag building set on $I$. Suppose that $I \in B$ has a binary decomposition $\mathcal{D}$. The two maximal elements $D_1,~D_2 \in \mathcal{D}-\{I\}$ with respect to inclusion are the \emph{maximal components} of $I$ in $\mathcal{D}$. The following lemma gives another definition of when a building set is flag.

\begin{lemma}
A building set $\mathcal{B}$ is flag if and only if every non singleton $I \in \mathcal{B}$ has a binary decomposition. \label{flaglem}
\end{lemma}

\begin{proof}
The only if part follows immediately from \cite[Proposition 7.3]{prw}.\\

For the if part, suppose that $\mathcal{B}$ is a building set and every element has a binary decomposition. We show that $\mathcal{B}$ is flag by showing that part (3) of Proposition \ref{flagprop} holds. Suppose by contradiction that (3) does not hold so that there exists a set $\mathcal{S} = \{S_1,...,S_k\} \subset \mathcal{B}$, $k \ge 3$, such that $S_i \cap S_j = \emptyset$, $S_i \cup S_j \not \in \mathcal{B}$ for all $i \ne j$, and $S_1 \cup \cdots \cup S_k = I \in \mathcal{B}$. Fix a decomposition $\mathcal{D}$ of $I$. Now consider all one element sets of $\mathcal{D}$ (the set of all $\{i\}$ such that $i \in I$). They are each a subset of one element of $\mathcal{S}$. Suppose by induction that all elements in $\mathcal{D}$ that are sets with $\le i$ elements are a subset of one element of $\mathcal{S}$. Then any $i+1$ element subset of $\mathcal{D}$ must also be contained in one element of $\mathcal{S}$. This is true since each $i+1$ element subset of $\mathcal{D}$ is the union of two elements of $\mathcal{D}$ each with less than $i+1$ elements. These two subsets must be contained in the same element of $\mathcal{S}$ since if they were contained in two distinct elements then their union would intersect two elements $S_i$ and $S_j$ of $\mathcal{S}$ which implies $S_i \cup S_j \in \mathcal{B}$. As the size of the elements of the decomposition increase, they are eventually equal to $I$, which implies that $k=1$, a contradiction since $k \ge 3$.
\end{proof}

\begin{corol}
A building set $\mathcal{B}$ is flag if and only if for every non singleton $I \in \mathcal{B}$, there exists two elements $D_1,~D_2 \in \mathcal{B}$ such that $D_1 \cap D_2 = \emptyset$ and $D_1 \cup D_2 =I.$
\end{corol}

\begin{lemma}
Suppose $\mathcal{B}$ is a flag building set. If $I,J \in \mathcal{B}$ and $J \subsetneq I$, then there is a decomposition of $I$ in $\mathcal{B}$ that contains $J$. \label{decomp}
\end{lemma}
\begin{proof}
Consider the set $\{J ,\{i_1\},...,\{i_k\}\}$ where $\{i_1,...,i_k\} = I-J$. This is a set of disjoint elements whose union is in $\mathcal{B}$. Therefore, by Proposition \ref{flagprop} part (2) we can reindex these sets until we obtain two disjoint sets each in $\mathcal{B}$ whose union is $I$. We can repeatedly perform this same procedure on the elements in $\{J,\{i_1\},\{i_2\},...,\{i_k\}\}$ that are subsets of each of the new sets obtained at each step. All of the new sets obtained with reindexing, together with a decomposition of $J$, and the element $I$ are a decomposition of $I$ that contains $J$.
\end{proof}

\end{section}

\begin{section}{Face shavings of flag building sets}
%Rename this section?

The following Theorem is proven by Volodin \cite{vol}.

\begin{thm}\cite[Lemma 6]{vol}.
Let $\mathcal{B}$ and $\mathcal{B}'$ be connected flag building sets on $[n]$ such that $\mathcal{B} \subseteq \mathcal{B}'$. Then $\mathcal{B}'$ can be obtained from $\mathcal{B}$ by successively adding elements so that at each step the set is a flag building set. \label{flagbuild}
\end{thm}

Suppose that a connected flag building set $\mathcal{B}'$ on $[n]$ is obtained from a flag building set $\mathcal{B}$ on $[n]$ by adding an element $I$. Then $I$ has a binary decomposition in $\mathcal{B}'$ with two maximal components $D_1$, $D_2$. This implies that $P_{\mathcal{B}'}$ can be obtained by shaving the codimension 2 face of $P_{\mathcal{B}}$ that corresponds to the nested set $\{D_1, D_2\}$.

\begin{lemma}
Let $\mathcal{B}$ be a building set with nestohedron $P_{\mathcal{B}}$. Suppose that $F_0$ is a facet of $P_{\mathcal{B}}$ corresponding to a (non-maximal) building set element $I$. Then the face poset of $F_0$ is isomorphic to the poset of faces of $P_{\mathcal{B}|_I} \times P_{\mathcal{B}/I}$. \label{equivlem}
\end{lemma}
\begin{proof}
The poset of faces of $F_0$ is the subposet of the faces of $P$, consisting of faces that are contained in $F_0$. Since the facet $F_0$ corresponds to the nested set $\{I\}$, the set of faces of $P$ that are contained in $F_0$ correspond to nested sets that contain $I$. The complex of nested sets of $\mathcal{B}$ that contain $I$ is isomorphic to $\Delta_{\mathcal{B}|_I} \times \Delta_{\mathcal{B}/I}$. The isomorphism is given by

$$(N_1,N_2) \in \Delta_{\mathcal{B}|_I} \times \Delta_{\mathcal{B}/I} \mapsto N_1 \cup N_2' \cup \{I\},$$
where $N_2':= \{D~|~D \in N_2~\hbox{and}~D \cup I \not \in \mathcal{B}\} \cup \{D \cup I~|~D \in N_2,~D \cup I \in \mathcal{B}\}.$ It is not too hard to see that this is a map to nested sets that contain $I$, that preserves the inclusion relation, and that is injective and surjective.\\

\end{proof}

\cite[Proposition 5]{vol} states that if a polytope $Q$ can be obtained from a simple $n$-dimensional polytope $P$ by shaving a face $G$ of dimension $k$ to obtain a new facet $F_0$, then $F_0$ is combinatorially equivalent to $G \times \Delta^{n-k-1}$, where $\Delta^d$ denotes the $d$-dimensional simplex. If $G$ is of dimension $n-2$ then $F_0$ is combinatorially equivalent to $G \times \Delta^{1}$, so that $\gamma(F_0) = \gamma(G)\gamma(\Delta^1) =\gamma(G)$. Hence, in the case that the polytopes are flag nestohedra, using Lemma \ref{equivlem}, we can rewrite \cite[Corollary 1]{vol} as:

\begin{lemma}\cite[Corollary 1]{vol}.
If $\mathcal{B}'$ is a flag building set on $[n]$ obtained from a flag building set $\mathcal{B}$ on $[n]$ by adding an element $I$ then

\begin{align*}\gamma(\mathcal{B}') =& \gamma(\mathcal{B}) + t\gamma(\mathcal{B}'|_I)\gamma(\mathcal{B}'/I)\\
=& \gamma(\mathcal{B}) + t\gamma(\mathcal{B}|_I)\gamma(\mathcal{B}/I).\end{align*}
\label{vollem}
\end{lemma}

\begin{proof}
The first identity is a direct consequence of the preceding discussion. From the definition of the contraction of a building set we have that $\mathcal{B}'/I = \mathcal{B}/I$ so that $\gamma(\mathcal{B}'/I) = \gamma(\mathcal{B}/I)$. Let $D_1,D_2$ be the maximal components of $I$ in the decomposition of $I$ in $\mathcal{B}'$. They are unique since $I \not \in \mathcal{B}$. Using Lemma \ref{combeq} below we have that $\mathcal{B}'|_I = \mathcal{D}[\mathcal{B}|_{D_1},\mathcal{B}|_{D_2}]$ where $\mathcal{D}$ is the building set $\{\{1\},\{2\},[2]\}$. Hence $$\gamma(\mathcal{B}'|_I) = \gamma(\mathcal{D})\gamma(\mathcal{B}|_{D_1})\gamma(\mathcal{B}|_{D_2}) = \gamma(\mathcal{D})\gamma(\mathcal{B}|_I) = \gamma(\mathcal{B}|_I).$$
\end{proof}

\noindent Note that if $\mathcal{B}$ is a flag building set on $[n]$ and $I \in \mathcal{B}$, then $\mathcal{B}/I$ and $\mathcal{B}|_I$ are flag building sets. This is obvious for $\mathcal{B}|_I$. For the claim about $\mathcal{B}/I$, we let $D \in \mathcal{B}/I$. Then if $D \in \mathcal{B}$ there exist two elements $D_1,D_2$ in $\mathcal{B}/I$ such that $D_1 \cap D_2 = \emptyset$ and $D_1 \cup D_2 = I$. If $D \not \in \mathcal{B}$ then $D \cup I \in \mathcal{B}$, and since $I \subseteq I \cup D$, by Lemma \ref{decomp}, $I$ is in a decomposition $\mathcal{D}$ of $I \cup D$ and this implies there are two elements $D_1, D_2 \in \mathcal{D}$ such that $D_1 \cap D_2 = \emptyset$, $D_1 \cup D_2 = D \cup I$, and $I$ is a proper subset of either $D_1$ or $D_2$. Let $\overline{D_i}$ denote the image of $D_i$ in the contraction. Then $\overline{ D_1} \cap \overline{ D_2} = \emptyset$ and $\overline{D_1} \cup \overline{D_2} = D$.\\

Using Theorem \ref{flagbuild} and Lemma \ref{vollem} \cite{vol} shows the following two Theorems. Their proof uses the inductive hypothesis that both $\gamma(\mathcal{B}'|_I)$ and $\gamma(\mathcal{B}'/I)$ of Lemma \ref{vollem} are such that $\gamma(\mathcal{B}'|_I) \ge 0$ and $\gamma(\mathcal{B}'/I) \ge 0$.\\

\begin{thm}\cite[Theorem 2]{vol}.
For any flag nestohedron $P_{\mathcal{B}}$ we have $$\gamma(\mathcal{B}) \ge 0.$$ \label{posthm}
\end{thm}

\begin{thm}\cite[Theorem 3]{vol} \cite[Theorem 1.1]{bv}.
If $\mathcal{B}$ and $\mathcal{B}'$ are connected flag building sets on $[n]$ and $\mathcal{B} \subseteq \mathcal{B}'$, then $\gamma(\mathcal{B}) \le \gamma(\mathcal{B}')$. \label{volthm}
\end{thm}

The following construction is due to Erokhovets \cite{Er}. Let $[i,j]$ denote the interval $\{i,i+1,...,j\}$. Let $\mathcal{B},\mathcal{B}_1,\mathcal{B}_2,...,\mathcal{B}_n$ be connected building sets on $[n], [k_1],...,[k_n]$ respectively, and let $\overline {[k_i]}$ denote the interval $[\sum_{j=1}^{i-1}k_j+1,\sum_{j=1}^{i}k_j]$. Define the connected building set $\mathcal{B}[\mathcal{B}_1,\mathcal{B}_2,...,\mathcal{B}_n]$ on $[k_1 + k_2 + \cdots +k_n]$, where $\mathcal{B}|_{\overline {[k_i]}}$ is equivalent to $\mathcal{B}_i$, and add the elements $\overline{[k_{i_1}] } \cup \overline{ [k_{i_2}]} \cup \cdots  \cup \overline{ [k_{i_m}]}$ for every $\{i_1,i_2,....,i_m\} \in \mathcal{B}.$

\begin{lemma}
\cite{Er}. Let $\mathcal{B},\mathcal{B}_1,...,\mathcal{B}_n$ be connected building sets on $[n],[k_1],...,[k_n]$ respectively. Let $\mathcal{B}' =\mathcal{B}[\mathcal{B}_1,...,\mathcal{B}_n]$. Then
$P_{\mathcal{B}'}$ is combinatorially equivalent to $P_\mathcal{B} \times \mathcal{P}_{B_1} \times \cdots \times P_{\mathcal{B}_n}$. \label{combeq}
\end{lemma}

\end{section}
%\widetilde{}.
\begin{section}{Tree shifts}
Our goal of this section is to prove Theorem \ref{bigthm}.\\

We will now introduce the tree shift operation mentioned in Theorem \ref{bigthm}. We call a degree one vertex of an arbitrary  graph a leaf (this is the standard name for a degree one vertex of a tree).  \\

Let $G$ be a connected graph with $n$ vertices labelled 1 to $n$, with the following properties and extra data (for a vertex $v$ we also denote the set $\{v\}$ by $v$):
\begin{enumerate}
\item $G$ has a leaf $l$ and the nearest vertex to $l$ of degree greater than 2 is labelled $c$. The vertices in the path from $c$ to $l$ are labelled $c,c_1,c_2,...,c_k,l$.\\
\item There exists a set of vertices $F$ of $G-\{c,c_1,...,l\}$ such that $F \cup c$ is a subgraph of $G$ that forms a tree, and such that there is no vertex of $G-(c \cup F)$ that is connected to a vertex in $F$.\\
\item $G - (F \cup \{c , c_1, c_2,...,c_k, l\}) \ne \emptyset$, and is denoted $E$.\\
\end{enumerate}

\noindent A \emph{tree shift} is the following move applied to a graph with the properties described. Informally, we remove $F$ and reattach $F$ to $l$. More formally, we remove any edge $(v,c)$ where $v \in F$, and replace it with the edge $(v,l)$ (see Figure \ref{treshift}). \\

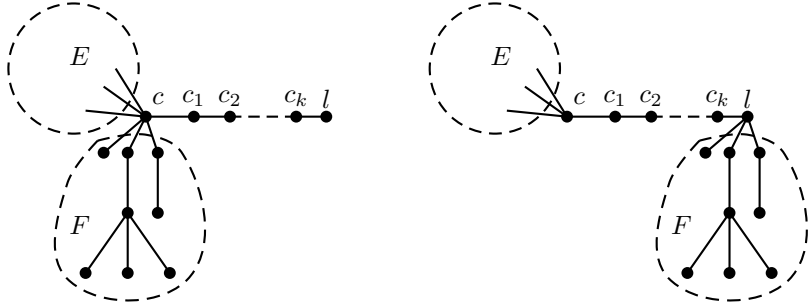
\begin{figure}[H]
\caption{A graph $G$ followed by the tree shift of $G$.}
\label{treshift}

\[
\psset{unit=0.8cm}
\begin{pspicture}(4,0)(10,4.5)
\rput(5.9,4){$E$}
%\rput(8.8,3.4){$b$}\qdisk(8.8,3){2.25pt}
\rput(7.2,3.3){$c$}\qdisk(7,3){2.25pt}%\rput(7.4,1.6){$e$}\qdisk(8,1.6){2.25pt}
%\rput(6.6,.4){$g$}\qdisk(7.2,.4){2.25pt}\rput(8.8,0){$h$}\qdisk(8.8,.4){2.25pt}
%\psline(8.8,3)(7,3)
\psline(7,3)(6.5,3.8)
\psline(7,3)(6.3,3.5)
\psline(7,3)(6, 3.1)
\psline(7,3)(8.4,3)\psline[linestyle=dashed](8.4,3)(9.5,3)\psline(9.5,3)(10,3)
\rput(7.8,3.3){$c_1$}\rput(8.4,3.3){$c_2$}\rput(9.5,3.3){$c_k$}\rput(10,3.3){$l$}
\qdisk(7.8,3){2.25pt}\qdisk(8.4,3){2.25pt}\qdisk(9.5,3){2.25pt}\qdisk(10,3){2.25pt}
\pscircle[linestyle=dashed](5.8,3.8){1.1}
\psline(7,3)(6.3,2.4)\qdisk(6.3,2.4){2.25pt}
\psline(7,3)(6.7,2.4)(6.7,1.4)(6.7,.4)\qdisk(6.7,2.4){2.25pt}\qdisk(6.7,1.4){2.25pt}\qdisk(6.7,.4){2.25pt}
\psline(7,3)(7.2,2.4)(7.2,1.4)\qdisk(7.2,2.4){2.25pt}\qdisk(7.2,1.4){2.25pt}
\psline(6.7,1.4)(6,.4)\qdisk(6,.4){2.25pt}
\psline(6.7,1.4)(7.4,.4)\qdisk(7.4,.4){2.25pt}
\pscurve[linestyle=dashed](6.2,2.6)(5.7,1.9)(5.7,.3)(7.7,.3)(7.6,2.4)(7,2.7)(6.2,2.6)\rput(5.9,1.2){$F$}
%original location (6.6,0)(14.2,4.5)
\end{pspicture}
\begin{pspicture}(3,0)(14.2,4.5)
\rput(5.9,4){$E$}
%\rput(8.8,3.4){$b$}\qdisk(8.8,3){2.25pt}
\rput(7.2,3.3){$c$}\qdisk(7,3){2.25pt}%\rput(7.4,1.6){$e$}\qdisk(8,1.6){2.25pt}
%\rput(6.6,.4){$g$}\qdisk(7.2,.4){2.25pt}\rput(8.8,0){$h$}\qdisk(8.8,.4){2.25pt}
%\psline(8.8,3)(7,3)
\psline(7,3)(6.5,3.8)
\psline(7,3)(6.3,3.5)
\psline(7,3)(6, 3.1)
\psline(7,3)(8.4,3)\psline[linestyle=dashed](8.4,3)(9.5,3)\psline(9.5,3)(10,3)
\rput(7.8,3.3){$c_1$}\rput(8.4,3.3){$c_2$}\rput(9.5,3.3){$c_k$}\rput(10,3.3){$l$}
\qdisk(7.8,3){2.25pt}\qdisk(8.4,3){2.25pt}\qdisk(9.5,3){2.25pt}\qdisk(10,3){2.25pt}
\pscircle[linestyle=dashed](5.8,3.8){1.1}

\psline(10,3)(9.3,2.4)\qdisk(9.3,2.4){2.25pt}
\psline(10,3)(9.7,2.4)(9.7,1.4)(9.7,.4)\qdisk(9.7,2.4){2.25pt}\qdisk(9.7,1.4){2.25pt}\qdisk(9.7,.4){2.25pt}
\psline(10,3)(10.2,2.4)(10.2,1.4)\qdisk(10.2,2.4){2.25pt}\qdisk(10.2,1.4){2.25pt}
\psline(9.7,1.4)(9,.4)\qdisk(9,.4){2.25pt}
\psline(9.7,1.4)(10.4,.4)\qdisk(10.4,.4){2.25pt}
\pscurve[linestyle=dashed](9.2,2.6)(8.7,1.9)(8.7,.3)(10.7,.3)(10.6,2.4)(10,2.7)(9.2,2.6)\rput(8.9,1.2){$F$}

%original location (6.6,0)(14.2,4.5)
\end{pspicture}
\]
\end{figure}
\begin{thm}
Let $G$ be a connected graph, and let $G'$ be a resulting tree shift of $G$. Then $\gamma(\mathcal{B}(G')) \le \gamma(\mathcal{B}(G))$. \label{bigthm}
\end{thm}
\begin{proof}
We suppose that $G$ has $n$ vertices, and we label $G$ as in the definition of a tree shift. We assume by induction that for any connected graph $H$ with less than $n$ vertices, if $H'$ is a tree shift of $H$, then $\gamma(\mathcal{B}(H')) \le \gamma(\mathcal{B}(H))$. When $n <4$ no tree shift is possible so the result is vacuously true. Let  $v$ be a leaf of $G$ (and $G'$) contained in $F$. The set $\overline{\mathcal{B}}:=\mathcal{B}(G-v) \cup  \{\{v\} , [n]\}$ is a flag building set contained in $\mathcal{B}(G)$ and $\overline{\mathcal{B}'}=\mathcal{B}(G'-v) \cup \{ \{v\} , [n]\}$ is a flag building set contained in $\mathcal{B}(G')$, hence, by Theorem \ref{flagbuild} we can add elements to $\overline{\mathcal{B}}$ to obtain $\mathcal{B}(G)$ so that at each step the set obtained is a flag building set. Similarly, we can add elements to $\overline{\mathcal{B}'}$ to obtain $\mathcal{B}(G')$ so that at each step the set we obtain is a flag building set. By Lemma \ref{vollem} and Theorem \ref{posthm} each time an element is added to these flag building sets the $\gamma$-polynomial of the resulting building set increases. We will construct an injection $$\mathcal{B}(G') - \overline{\mathcal{B}}' \rightarrow \mathcal{B}(G) - \overline{\mathcal{B}}$$ $$I' \mapsto I,$$ and show that the increase in the $\gamma$-polynomial when adding $I'$ is less than or equal to the increase when adding $I$. This shows that

\begin{equation}\gamma(\mathcal{B}(G'))- \gamma(\overline{\mathcal{B}'}) \le  \gamma(\mathcal{B}(G)) - \gamma(\overline{\mathcal{B}}).\label{lemequn}\end{equation}

\noindent By Lemma \ref{combeq} $$\gamma(\overline{\mathcal{B}}) = \gamma(\mathcal{B}(G-v))$$ and

$$\gamma(\overline{\mathcal{B}'}) = \gamma(\mathcal{B}(G'-v)),$$ so that Equation \ref{lemequn} becomes

$$\gamma(\mathcal{B}(G'))- \gamma(\mathcal{B}(G'-v)) \le  \gamma(\mathcal{B}(G)) - \gamma(\mathcal{B}(G-v)).$$ By induction, since $G'-v$ is a tree shift of $G-v$, or is equal to $G-v$, we have

$$\gamma(\mathcal{B}(G'-v)) \le \gamma(\mathcal{B}(G-v))$$ so that

$$\gamma(\mathcal{B}(G')) \le \gamma(\mathcal{B}(G)).$$

We will now construct the injection. Suppose that $I_1',I_2',...,I_k'$ are the building set elements that are added to $\overline{\mathcal{B}}'$ to obtain $\mathcal{B}(G')$ (in order) and $I_j' \subseteq I_i'$. Then $j > i$, since $I_j' \cap (I_i' - \{v\}) \ne \emptyset$ and $I_j' \cup (I_i' - \{v\}) = I_i'$ which implies that when $I_j'$ is in the building set $I_i'$ must be too. Similarly, no subset of an element is added before it when we are adding sets to obtain $\mathcal{B}(G)$.\\

Let $\mathcal{B}_m'$ be the building set $\overline {\mathcal{B}}' \cup \{I_{1}', I_{2}',..., I_{m}'\}$. By Lemma \ref{vollem} we have that $$\gamma(\mathcal{B}_m') -\gamma(\mathcal{B}_{m-1}') = t\gamma(\mathcal{B}_{m-1}'|_{I_{m}'})\gamma(\mathcal{B}_{m-1}'/I_m').$$ Suppose that $I_m' \cap E = \emptyset$, so that $I_m' = D \cup \{l,c_k, ..., c_{k-\alpha+1}\}$ for some $D \subseteq F$ and let $I_m = D \cup \{c,c_1,...,c_\alpha\}$, one of the elements that is added to $\overline{ \mathcal{B}}$ to obtain $\mathcal{B}(G)$. Note that we may have $c_{k-\alpha + 1} = c$ and $c_\alpha = l$. Note also that $I_m$ is not necessarily the $m$th element that is added to $\overline{\mathcal{B}}$ (see Figure \ref{setbm}). \\

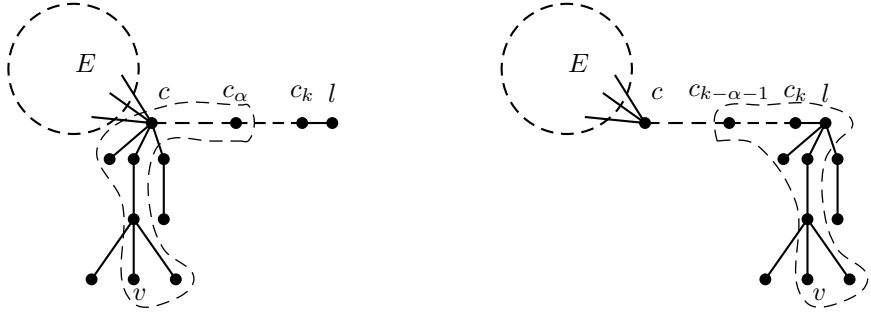
\begin{figure}[H]

\caption{The set $I_m$ followed by the set $I_m'$.}
\label{setbm}

\[
\psset{unit=0.8cm}
\begin{pspicture}(4,0)(14.2,5)
\rput(5.9,4){$E$}
%\rput(8.8,3.4){$b$}\qdisk(8.8,3){2.25pt}
\rput(7.2,3.5){$c$}\qdisk(7,3){2.25pt}%\rput(7.4,1.6){$e$}\qdisk(8,1.6){2.25pt}
%\rput(6.6,.4){$g$}\qdisk(7.2,.4){2.25pt}\rput(8.8,0){$h$}\qdisk(8.8,.4){2.25pt}
%\psline(8.8,3)(7,3)
\psline(7,3)(6.5,3.8)
\psline(7,3)(6.3,3.5)
\psline(7,3)(6, 3.1)
\psline[linestyle=dashed](7,3)(8.4,3)\psline[linestyle=dashed](8.4,3)(9.5,3)\psline(9.5,3)(10,3)
\rput(8.4,3.5){$c_\alpha$}\rput(9.5,3.5){$c_k$}\rput(10,3.5){$l$}
\qdisk(8.4,3){2.25pt}\qdisk(9.5,3){2.25pt}\qdisk(10,3){2.25pt}
\pscircle[linestyle=dashed](5.7,3.9){1.1}
\psline(7,3)(6.3,2.4)\qdisk(6.3,2.4){2.25pt}
\psline(7,3)(6.7,2.4)(6.7,1.4)(6.7,.4)\qdisk(6.7,2.4){2.25pt}\qdisk(6.7,1.4){2.25pt}\qdisk(6.7,.4){2.25pt}
\psline(7,3)(7.2,2.4)(7.2,1.4)\qdisk(7.2,2.4){2.25pt}\qdisk(7.2,1.4){2.25pt}
\psline(6.7,1.4)(6,.4)\qdisk(6,.4){2.25pt}
\psline(6.7,1.4)(7.4,.4)\qdisk(7.4,.4){2.25pt}

\pscurve[linestyle=dashed, linewidth=.5pt](8.5,2.7)(8.6,2.7)(8.6,3.3)(8.5,3.3)(7.3,3.3)(6.1,2.6)(6.5,1.7)(6.7,0)(7.2,0)
(7.7,.4)(7.2,.9)(7,1.4)(7,2.4)(7.2,2.7)(8,2.7)(8.5,2.7)
\rput(6.8, 0.15){$v$}
%(6.6,0)(14.2,4.5)
\end{pspicture}
\begin{pspicture}(6,0)(14.2,5)
\rput(5.9,4){$E$}
%\rput(8.8,3.4){$b$}\qdisk(8.8,3){2.25pt}
\rput(7.2,3.5){$c$}\qdisk(7,3){2.25pt}%\rput(7.4,1.6){$e$}\qdisk(8,1.6){2.25pt}
%\rput(6.6,.4){$g$}\qdisk(7.2,.4){2.25pt}\rput(8.8,0){$h$}\qdisk(8.8,.4){2.25pt}
%\psline(8.8,3)(7,3)
\psline(7,3)(6.5,3.8)
\psline(7,3)(6.3,3.5)
\psline(7,3)(6, 3.1)
\psline[linestyle=dashed](7,3)(8.4,3)\psline[linestyle=dashed](8.4,3)(9.5,3)\psline(9.5,3)(10,3)
\rput(8.4,3.5){$c_{k-\alpha-1}$}\rput(9.5,3.5){$c_k$}\rput(10,3.5){$l$}
\qdisk(8.4,3){2.25pt}\qdisk(9.5,3){2.25pt}\qdisk(10,3){2.25pt}
\pscircle[linestyle=dashed](5.7,3.9){1.1}

\psline(10,3)(9.3,2.4)\qdisk(9.3,2.4){2.25pt}
\psline(10,3)(9.7,2.4)(9.7,1.4)(9.7,.4)\qdisk(9.7,2.4){2.25pt}\qdisk(9.7,1.4){2.25pt}\qdisk(9.7,.4){2.25pt}
\psline(10,3)(10.2,2.4)(10.2,1.4)\qdisk(10.2,2.4){2.25pt}\qdisk(10.2,1.4){2.25pt}
\psline(9.7,1.4)(9,.4)\qdisk(9,.4){2.25pt}
\psline(9.7,1.4)(10.4,.4)\qdisk(10.4,.4){2.25pt}
\pscurve[linestyle=dashed, linewidth=.5pt](8.2,2.7)(8.7,2.7)(8.9,2.6)(9.5,1.7)(9.7,0)(10.2,0)
(10.7,.4)(10.2,.9)(10,1.4)(10,2.4)(10.2,2.7)(10.45,3)(9.9,3.3)(9,3.3)(8.2,3.2)(8.2,2.7)%(11,2.7)(11.5,2.7)
%\pscurve[linestyle=dashed](9.2,2.6)(8.7,1.9)(8.7,.3)(10.7,.3)(10.6,2.4)(10,2.7)(9.2,2.6)\rput(8.9,1.2){$s$}
\rput(9.9, 0.15){$v$}
%(6.6,0)(14.2,4.5)
\end{pspicture}
\]
\end{figure}

\noindent We let $\mathcal{B}_m$ denote the building set obtained after adding the elements up to and including $I_m$ to $\overline {\mathcal{B}}$. Let $\widetilde{\mathcal{B}}_{m-1}$ denote the building set $\mathcal{B}_m -\{I_m\}$ (note that $\widetilde{\mathcal{B}}_{m-1}$ is not necessarily equal to $\mathcal{B}_{m-1}$ since $I_{m-1}$ is not necessarily added directly before $I_m$). Then by Lemma \ref{vollem} $$\gamma(\mathcal{B}_m) -\gamma(\widetilde{\mathcal{B}}_{m-1}) = t\gamma(\widetilde{\mathcal{B}}_{m-1}|_{I_m})\gamma(\widetilde{\mathcal{B}}_{m-1}/I_m).$$ Since we do not add a subset of a set before adding the set, we have that $$\widetilde{\mathcal{B}}_{m-1}|_{I_m} = \mathcal{B}(G)|_{I_m -\{v\}} \cup \{\{v\}\} \cong \mathcal{B}(G')|_{I_m' - \{v\}} \cup \{\{v\}\} = \mathcal{B}_{m-1}'|_{I_m'}.$$ We let $K'$ denote the set of vertices in $G' - I_m'$ that are adjacent in $G'$ to a vertex in $I_m'$, and we let $K$ denote the set of vertices in $G-I_m$ that are adjacent in $G$ to a vertex in $I_m$. Then $\mathcal{B}_{m-1}'/I_m' = \mathcal{B}(G')/I_m'$. This is true since we know that $\mathcal{B}_{m-1}'/I_m' \subseteq \mathcal{B}(G')/I_m'$ since $\mathcal{B}_{m-1}' \subseteq \mathcal{B}(G')$. To show that $\mathcal{B}_{m-1}'/I_m' \supseteq \mathcal{B}(G')/I_m'$, note that $\mathcal{B}(G')/I_m' = \mathcal{B}(\hat G')$ where $\hat G'$ is the graph $G'-I_m'$ with additional edges so that the restriction to $K'$ is a complete graph. The elements of $\mathcal{B}(\hat G')$ that are the edges between elements in $K'$ are in $\mathcal{B}_{m-1}'/I_m'$ because any two vertices in $K'$ are linked by a path of vertices contained in $I_m'-v$. By a similar argument we have that $\widetilde{\mathcal{B}}_{m-1}/I_m = \mathcal{B}(G)/I_m$. Note that $\mathcal{B}(G)/I_m = \mathcal{B}(\hat G)$ where $\hat G$ denotes the graph $G-I_m$ with additional edges so that the restriction to $K$ is a complete graph, (see Figure \ref{diagramfive}).\\

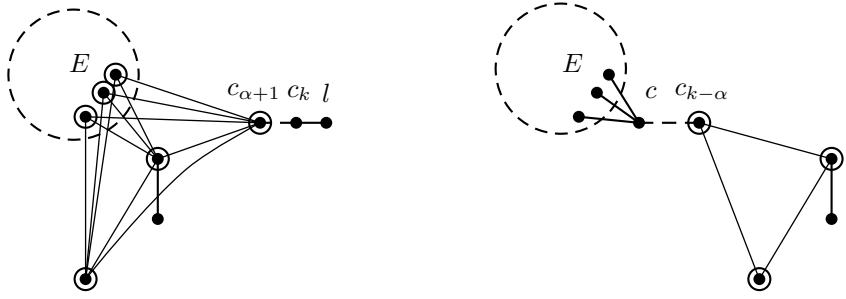
\begin{figure}[H]
\caption{The graph $\hat G$ for the contraction $\mathcal{B}_{m-1}/I_m =\mathcal{B}(\hat G)$ followed by the graph $\hat{G'}$ for the contraction $\mathcal{B}_{m-1}'/I_m' = \mathcal{B}(\hat{G'})$. The vertices $K$ and $K'$ are drawn with an additional ring around them.}
\label{diagramfive}
\[
\psset{unit=0.8cm}
\begin{pspicture}(4,0)(14.2,5)
\rput(5.9,4){$E$}
%\rput(8.8,3.4){$b$}\qdisk(8.8,3){2.25pt}
%\rput(7.2,3.5){$c$}\qdisk(7,3){2.25pt}%\rput(7.4,1.6){$e$}\qdisk(8,1.6){2.25pt}
%\rput(6.6,.4){$g$}\qdisk(7.2,.4){2.25pt}\rput(8.8,0){$h$}\qdisk(8.8,.4){2.25pt}
%\psline(8.8,3)(7,3)
\qdisk(6.5,3.8){2.25pt}
\qdisk(6.3,3.5){2.25pt}
\qdisk(6,3.1){2.25pt}
\psline[linestyle=dashed](8.8,3)(9.5,3)\psline(9.5,3)(10,3)
\rput(8.8,3.5){$c_{\alpha+1}$}\rput(9.55,3.5){$c_k$}\rput(10,3.5){$l$}
\qdisk(8.9,3){2.25pt}\qdisk(9.5,3){2.25pt}\qdisk(10,3){2.25pt}
\pscircle(6,3.1){.2}\pscircle(8.9,3){.2}\pscircle(6,.4){.2}\pscircle(6.5,3.8){.2}\pscircle(6.3,3.5){.2}\pscircle(7.2,2.4){.2}
\pscircle[linestyle=dashed](5.8,3.8){1.1}
%\psline(7,3)(6.3,2.4)\qdisk(6.3,2.4){2.25pt}
%\qdisk(6.7,2.4){2.25pt}\qdisk(6.7,1.4){2.25pt}%\qdisk(6,.4){2.25pt}
\psline(7.2,2.4)(7.2,1.4)\qdisk(7.2,2.4){2.25pt}\qdisk(7.2,1.4){2.25pt}
%\psline(6.7,1.4)(6,.4)
%\psline(6.7,1.4)(7.4,.4)\qdisk(7.4,.4){2.25pt}
\qdisk(6,.4){2.25pt}
\psline[linewidth=.5pt](6,.4)(7.2,2.4)
\psline[linewidth=0.5pt](6,.4)(6.5,3.8)
\psline[linewidth=0.5pt](6,.4)(6.3,3.5)
\psline[linewidth=0.5pt](6,.4)(6,3.1)
\psline[linewidth=0.5pt](8.9,3)(6.5,3.8)
\psline[linewidth=0.5pt](8.9,3)(6.3,3.5)
\psline[linewidth=0.5pt](8.9,3)(6,3.1)
\psline[linewidth=0.5pt](7.2,2.4)(6,3.1)
\psline[linewidth=0.5pt](7.2,2.4)(6.5,3.8)
\psline[linewidth=0.5pt](7.2,2.4)(6.3,3.5)
\pscurve[linewidth=0.5pt](6,.4)(7.7,2.25)(8.9,3)
\psline[linewidth=0.5pt](7.2,2.4)(8.9,3)
%\pscurve[linestyle=dashed, linewidth=.5pt](8.5,2.7)(8.6,2.7)(8.6,3.3)(8.5,3.3)(7.3,3.3)(6.1,2.6)(6.5,1.7)(6.7,0)(7.2,0)
%(7.7,.4)(7.2,.9)(7,1.4)(7,2.4)(7.2,2.7)(8,2.7)(8.5,2.7)
%\rput(6.8, 0.15){$v$}
%(6.6,0)(14.2,4.5)
\end{pspicture}
\begin{pspicture}(6,0)(14.2,5)
\rput(5.9,4){$E$}
%\rput(8.8,3.4){$b$}\qdisk(8.8,3){2.25pt}
\rput(7.2,3.5){$c$}\qdisk(7,3){2.25pt}%\rput(7.4,1.6){$e$}\qdisk(8,1.6){2.25pt}
%\rput(6.6,.4){$g$}\qdisk(7.2,.4){2.25pt}\rput(8.8,0){$h$}\qdisk(8.8,.4){2.25pt}
%\psline(8.8,3)(7,3)
\qdisk(6.5,3.8){2.25pt}
\qdisk(6.3,3.5){2.25pt}
\qdisk(6,3.1){2.25pt}
\psline(7,3)(6.5,3.8)
\psline(7,3)(6.3,3.5)
\psline(7,3)(6, 3.1)
\psline[linestyle=dashed](7,3)(8.,3)%\psline[linestyle=dashed](8.4,3)(9.5,3)\psline(9.5,3)(10,3)
\rput(8.05,3.5){$c_{k-\alpha}$}%\rput(9.5,3.5){$c_k$}\rput(10,3.5){$l$}
\qdisk(8.,3){2.25pt}%\qdisk(9.5,3){2.25pt}\qdisk(10,3){2.25pt}
\pscircle[linestyle=dashed](5.7,3.9){1.1}
%\psline(10,3)(9.3,2.4)\qdisk(9.3,2.4){2.25pt}
%\psline(10,3)(9.7,2.4)(9.7,1.4)(9.7,.4)\qdisk(9.7,2.4){2.25pt}\qdisk(9.7,1.4){2.25pt}\qdisk(9.7,.4){2.25pt}
\psline(10.2,2.4)(10.2,1.4)\qdisk(10.2,2.4){2.25pt}\qdisk(10.2,1.4){2.25pt}
\qdisk(9,.4){2.25pt}
\pscircle(10.2,2.4){.2}\pscircle(8.,3){.2}\pscircle(9,.4){.2}
%\psline(9.7,1.4)(10.4,.4)\qdisk(10.4,.4){2.25pt}
%\pscurve[linestyle=dashed, linewidth=.5pt](8.2,2.7)(8.7,2.7)(8.9,2.6)(9.5,1.7)(9.7,0)(10.2,0)
%(10.7,.4)(10.2,.9)(10,1.4)(10,2.4)(10.2,2.7)(10.45,3)(9.9,3.3)(9,3.3)(8.2,3.2)(8.2,2.7)%(11,2.7)(11.5,2.7)
%\pscurve[linestyle=dashed](9.2,2.6)(8.7,1.9)(8.7,.3)(10.7,.3)(10.6,2.4)(10,2.7)(9.2,2.6)\rput(8.9,1.2){$s$}
%\rput(9.9, 0.15){$v$}
\psline[linewidth=0.5pt](9,.4)(10.2,2.4)
\psline[linewidth=0.5pt](8,3)(10.2,2.4)
\psline[linewidth=0.5pt](8,3)(9,.4)
%(6.6,0)(14.2,4.5)
\end{pspicture}
\]
\end{figure}

\noindent We also have that $\gamma(\mathcal{B}_{m-1}'/I_m') \le \gamma(\widetilde{\mathcal{B}}_{m-1}/I_m)$ because $\hat{G'}$ can be obtained from $\hat{G}$ by first removing edges (which lowers the $\gamma$-polynomial of the corresponding graphical building set by Theorem \ref{volthm}) and then performing a tree shift on a graph with fewer than $n$ vertices (or doing no tree shift in the case that $c_{\alpha} = c_k$ or $c_{\alpha} = l$), which we assume lowers the $\gamma$-polynomial (see Figure \ref{diagramsix}). Hence

\begin{align*}\gamma(\mathcal{B}_m') - \gamma(\mathcal{B}_{m-1}') &= t\gamma(\mathcal{B}_{m-1}'|_{I'_m})\gamma(\mathcal{B}_{m-1}'/I_m')\\ &\le t\gamma(\widetilde{\mathcal{B}}_{m-1}|_{I_m})\gamma(\widetilde{\mathcal{B}}_{m-1}/I_m)\\ &= \gamma(\mathcal{B}_{m}) - \gamma(\widetilde{\mathcal{B}}_{m-1}).\end{align*}

%$$\gamma(\mathcal{B}_m') - \gamma(\mathcal{B}_{m-1}') = t\gamma(\mathcal{B}_{m-1}'|_{I_m})\gamma(\mathcal{B}_{m-1}'/I_m) \le t\gamma(\widetilde{\mathcal{B}}_{m-1}|_{I_m})\gamma(\widetilde{\mathcal{B}}_{m-1}/I_m) = \gamma(\widetilde{\mathcal{B}}_{m}) - \gamma(\widetilde{\mathcal{B}}_{m-1}).$$

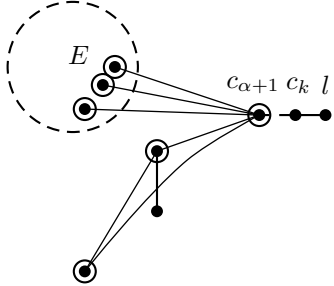
\begin{figure}[H]
\caption{The graph that is obtained after removing edges from $\hat G$ in Figure \ref{diagramfive}. The tree shift of this graph gives the graph $\hat{G'}$ of Figure \ref{diagramfive}.}
\label{diagramsix}
\[
\psset{unit=0.8cm}
\begin{pspicture}(1,0)(14.2,5)
\rput(5.9,4){$E$}
%\rput(8.8,3.4){$b$}\qdisk(8.8,3){2.25pt}
%\rput(7.2,3.5){$c$}\qdisk(7,3){2.25pt}%\rput(7.4,1.6){$e$}\qdisk(8,1.6){2.25pt}
%\rput(6.6,.4){$g$}\qdisk(7.2,.4){2.25pt}\rput(8.8,0){$h$}\qdisk(8.8,.4){2.25pt}
%\psline(8.8,3)(7,3)
\qdisk(6.5,3.8){2.25pt}
\qdisk(6.3,3.5){2.25pt}
\qdisk(6,3.1){2.25pt}
\psline[linestyle=dashed](8.8,3)(9.5,3)\psline(9.5,3)(10,3)
\rput(8.8,3.5){$c_{\alpha+1}$}\rput(9.55,3.5){$c_k$}\rput(10,3.5){$l$}
\qdisk(8.9,3){2.25pt}\qdisk(9.5,3){2.25pt}\qdisk(10,3){2.25pt}
\pscircle(6,3.1){.2}\pscircle(8.9,3){.2}\pscircle(6,.4){.2}\pscircle(6.5,3.8){.2}\pscircle(6.3,3.5){.2}\pscircle(7.2,2.4){.2}
\pscircle[linestyle=dashed](5.8,3.8){1.1}
%\psline(7,3)(6.3,2.4)\qdisk(6.3,2.4){2.25pt}
%\qdisk(6.7,2.4){2.25pt}\qdisk(6.7,1.4){2.25pt}%\qdisk(6.7,.4){2.25pt}
\psline(7.2,2.4)(7.2,1.4)\qdisk(7.2,2.4){2.25pt}\qdisk(7.2,1.4){2.25pt}
%\psline(6.7,1.4)(6,.4)
%\psline(6.7,1.4)(7.4,.4)\qdisk(7.4,.4){2.25pt}
\qdisk(6,.4){2.25pt}
\psline[linewidth=.5pt](6,.4)(7.2,2.4)
%\psline[linewidth=0.5pt](6,.4)(6.5,3.8)
%\psline[linewidth=0.5pt](6,.4)(6.3,3.5)
%\psline[linewidth=0.5pt](6,.4)(6,3.1)
\psline[linewidth=0.5pt](8.9,3)(6.5,3.8)
\psline[linewidth=0.5pt](8.9,3)(6.3,3.5)
\psline[linewidth=0.5pt](8.9,3)(6,3.1)
%\psline[linewidth=0.5pt](7.2,2.4)(6,3.1)
%\psline[linewidth=0.5pt](7.2,2.4)(6.5,3.8)
%\psline[linewidth=0.5pt](7.2,2.4)(6.3,3.5)
\pscurve[linewidth=0.5pt](6,.4)(7.7,2.25)(8.9,3)
\psline[linewidth=0.5pt](7.2,2.4)(8.9,3)
%\pscurve[linestyle=dashed, linewidth=.5pt](8.5,2.7)(8.6,2.7)(8.6,3.3)(8.5,3.3)(7.3,3.3)(6.1,2.6)(6.5,1.7)(6.7,0)(7.2,0)
%(7.7,.4)(7.2,.9)(7,1.4)(7,2.4)(7.2,2.7)(8,2.7)(8.5,2.7)
%\rput(6.8, 0.15){$v$}
\end{pspicture}
\]
\end{figure}
Now suppose that $I_m' \cap E \ne \emptyset$, so that $\{c, c_1,...,c_k,l\} \subseteq I_m'$. Let $I_m$ denote $I_m'$, which is a set that is also added to $\overline{\mathcal{B}}$ to obtain $\mathcal{B}(G)$ (see Figure \ref{bmpics}). Define $\mathcal{B}_{m-1}',~\widetilde{\mathcal{B}}_{m-1}$ as in the previous case. \\

\begin{figure}[H]
\caption{The set $I_m$ followed by the set $I_m'$.}
\label{bmpics}
\[
\psset{unit=0.8cm}
\begin{pspicture}(4,-.5)(14.2,5.5)
\rput(5.34,4.2){$e$}
%\rput(8.8,3.4){$b$}\qdisk(8.8,3){2.25pt}
\rput(7.2,3.6){$c$}\qdisk(7,3){2.25pt}%\rput(7.4,1.6){$e$}\qdisk(8,1.6){2.25pt}
%\rput(6.6,.4){$g$}\qdisk(7.2,.4){2.25pt}\rput(8.8,0){$h$}\qdisk(8.8,.4){2.25pt}
%\psline(8.8,3)(7,3)
\psline(7,3)(6.5,3.8)
\psline(7,3)(6.3,3.5)
\psline(7,3)(6, 3.1)
\psline[linestyle=dashed](7,3)(9.5,3)\psline(9.5,3)(10,3)
\rput(9.5,3.5){$c_k$}\rput(10,3.5){$l$}
\qdisk(9.5,3){2.25pt}\qdisk(10,3){2.25pt}
\pscircle[linestyle=dashed](5.7,3.9){1.1}
\psline(7,3)(6.3,2.4)\qdisk(6.3,2.4){2.25pt}
\psline(7,3)(6.7,2.4)(6.7,1.4)(6.7,.4)\qdisk(6.7,2.4){2.25pt}\qdisk(6.7,1.4){2.25pt}\qdisk(6.7,.4){2.25pt}
\psline(7,3)(7.2,2.4)(7.2,1.4)\qdisk(7.2,2.4){2.25pt}\qdisk(7.2,1.4){2.25pt}
\psline(6.7,1.4)(6,.4)\qdisk(6,.4){2.25pt}
\psline(6.7,1.4)(7.4,.4)\qdisk(7.4,.4){2.25pt}
\pscurve[linestyle=dashed, linewidth=.5pt](10.1,2.7)(10.2,3)(10.2,3.2)(9.5,3.3)(8.5,3.3)(7.3,3.3)(6.3,4.3)(5.5,4)(5.5,3)(6.3,2.8)(6.1,2.6)(6.5,1.7)(6.7,0)(7.2,0)
(7.7,.4)(7.2,.9)(7,1.4)(7,2.4)(7.2,2.7)(8,2.7)(8.5,2.7)(10.1,2.7)
\rput(6.8, 0.15){$v$}
%(6.6,0)(14.2,4.5)
\end{pspicture}
\psset{unit=0.8cm}
\begin{pspicture}(6,-.5)(14.2,5.5)
\rput(5.9,4){$e$}
%\rput(8.8,3.4){$b$}\qdisk(8.8,3){2.25pt}
\rput(7.3,3.6){$c$}\qdisk(7,3){2.25pt}%\rput(7.4,1.6){$e$}\qdisk(8,1.6){2.25pt}
%\rput(6.6,.4){$g$}\qdisk(7.2,.4){2.25pt}\rput(8.8,0){$h$}\qdisk(8.8,.4){2.25pt}
%\psline(8.8,3)(7,3)
\psline(7,3)(6.5,3.8)
\psline(7,3)(6.3,3.5)
\psline(7,3)(6, 3.1)
\psline[linestyle=dashed](7,3)(9.5,3)\psline(9.5,3)(10,3)
\rput(9.5,3.5){$c_k$}\rput(10,3.5){$l$}
\qdisk(9.5,3){2.25pt}\qdisk(10,3){2.25pt}
\pscircle[linestyle=dashed](5.7,3.9){1.1}
\psline(10,3)(9.3,2.4)\qdisk(9.3,2.4){2.25pt}
\psline(10,3)(9.7,2.4)(9.7,1.4)(9.7,.4)\qdisk(9.7,2.4){2.25pt}\qdisk(9.7,1.4){2.25pt}\qdisk(9.7,.4){2.25pt}
\psline(10,3)(10.2,2.4)(10.2,1.4)\qdisk(10.2,2.4){2.25pt}\qdisk(10.2,1.4){2.25pt}
\psline(9.7,1.4)(9,.4)\qdisk(9,.4){2.25pt}
\psline(9.7,1.4)(10.4,.4)\qdisk(10.4,.4){2.25pt}
\pscurve[linestyle=dashed, linewidth=.5pt](8.2,2.7)(8.6,2.7)(9.4,1.7)(9.7,0)(10.2,0)
(10.7,.4)(10.2,.9)(10,1.4)(10,2.4)(10.2,2.7)(10.45,3)(9.9,3.3)(9,3.3)(8.2,3.3)(7.3,3.3)(6.3,4.3)(5.5,4)(5.5,3)(6.3,2.7)
(7.3,2.7)(8.2,2.7)
\rput(9.9, 0.15){$v$}
%(6.6,0)(14.2,4.5)
\end{pspicture}
\]
\end{figure}

\noindent Then we have that $\widetilde{\mathcal{B}}_{m-1}|_{I_m} = \mathcal{B}_{m-1}'|_{I_m'}$ and $\widetilde{\mathcal{B}}_{m-1}/I_m = \mathcal{B}_{m-1}'/I_m'$ which are both equal to $\mathcal{B}(G)/I_m$. This can be shown by arguments similar to those used in the case where $I_m' \cap E = \emptyset$. Hence in this case we also have

\begin{align*}\gamma(\mathcal{B}_m') - \gamma(\mathcal{B}_{m-1}') &= t\gamma(\mathcal{B}_{m-1}'|_{I_m'})\gamma(\mathcal{B}_{m-1}'/I_m')\\ &\le t\gamma(\widetilde{\mathcal{B}}_{m-1}|_{I_m})\gamma(\widetilde{\mathcal{B}}_{m-1}/I_m)\\ &= \gamma(\mathcal{B}_{m}) - \gamma(\widetilde{\mathcal{B}}_{m-1}).\end{align*}

Since for every element $I_m'$ that is added to $\overline{\mathcal{B}'}$ to obtain $\mathcal{B}(G')$ there is a corresponding element $I_m$ that is added to $\overline{\mathcal{B}}$ to obtain $\mathcal{B}(G)$ that increases the $\gamma$-polynomial by at least as much as $I_m'$ we have that $$\gamma(\mathcal{B}(G')) - \gamma(\overline{\mathcal{B}'}) \le  \gamma(\mathcal{B}(G))- \gamma(\overline{\mathcal{B}})$$ as desired.

\end{proof}

% $$\gamma(B^{m-1}(B_1^{m-1 \prime },....,B_{m-1})) - \gamma(B^{m-1}(B_1^{m-1},....,B_{m-1})) = \gamma(B^{m-1})(\gamma(B_1^{m-1 \prime}) - \gamma(B_1^{m-1})).$$

By applying Theorem \ref{bigthm} to the case where the graph is a tree we obtain the following Theorem, which is predicted by \cite[Conjecture 14.1]{prw}.
\begin{thm}
Let $S$ be the set of all tree graphs on $n$ nodes. Define the relation $T'\le T$ if $T'$ can be obtained by applying any number of tree shifts to $T$. Then $\le$ defines a partial order on $S$ with the following properties.
\begin{itemize}
\item $\mathrm{Path_n}$ is the unique $\le$-minimum element.
\item $K_{1,n-1}$ is the unique $\le$-maximum element.
\item $T'\le T$ implies $\gamma (\mathcal{B}(T')) \le \gamma(\mathcal{B}(T))$.
\end{itemize}\label{partial}
\end{thm}

\begin{proof}
This relation is a partial order on $S$, since given any $a,b \in S$ we have that if $a \le  b$ and $b \le a$ then $a=b$ because any tree shift decreases the number of leaves by one.\\

$\mathrm{Path_n}$ is $\le$-minimal since no tree has fewer leaves than $\mathrm{Path_n}$. Let $T$ be a tree that is not $\mathrm{Path_n}$. We can apply a tree shift to $T$ since if we travel along the path from any leaf inwards we must eventually meet a vertex of degree three or more. Hence $T$ is not $\le$-minimal, so that $\mathrm{Path_n}$ is the unique $\le$-minimum element.\\

$K_{1,n-1}$ is $\le$-maximal because no tree has more leaves than $K_{1,n-1}$. Suppose that $T'$ is a tree that is not $K_{1,n-1}$. We can perform a \emph{reverse shift}, which sends $T'$ to a tree $T$ such that we can apply a tree shift to $T$ to obtain $T'$. $T'$ must contain two adjacent vertices $c$ and $l$, neither of which is a leaf. To obtain $T$, we attach the component of $T'-\{c ,l\}$ that was attached to $l$ in $T$, and attach it to $c$, so that the vertices that were attached to $l$ are now attached to $c$. Hence $T'$ is not $\le$-maximal, so that $K_{1,n-1}$ is the unique $\le $-maximum element.\\

By Theorem \ref{bigthm}, if $T' \le T$ then $\gamma(\mathcal{B}(T')) \le \gamma(\mathcal{B}(T)).$
\end{proof}

Theorem \ref{bigthm} provides a new (arguably more explicit) proof of the bounds on the $\gamma$-polynomial of trees (Equation \ref{eqone}) than that provided in \cite[Theorem 9.4, (1)]{bv}.

\end{section}

\begin{section}{Flossing moves}

Let $G$ be a graph with $n$ vertices labelled 1 to $n$. A pair of leaves $l, \hat l$ in $G$ \emph{floss} a vertex $v \in G$ if there is a unique path in $G$ from $l$ to $\hat l$ of minimal length, and $v$ is the unique branched vertex (having degree $\ge 3$) on this path. \cite[Proposition 4.8]{br} shows that for any tree graph $T$ that is not $\mathrm{Path_n}$, there exists a triple of vertices $(l,\hat l,v)$ in which the vertices $l,\hat l$ floss the vertex $v$. When $l, \hat l$ floss a vertex $v$, relabel so that $$\mathrm{dist}_G(l,v) \le \mathrm{dist}_G(\hat l,v),$$ where $\mathrm{dist}_G(v_1,v_2)$ denotes the number of edges in a minimal path in $G$ between vertices $v_1$ and $v_2$. Flossing moves are defined in \cite{br}, and it was suggested in \cite{prw} that they might lower the $\gamma$-polynomial of the graph-associahedra. We show that this is true for flossing moves that are a generalisation of those given in \cite{br}. Let $G$ be a graph with a triple of vertices $(l, \hat l, v)$ such that $l, \hat l$ are leaves that floss the vertex $v$ (and $\mathrm{dist}_G(l,v) \le \mathrm{dist}_G(\hat l,v)$). A \emph{flossing move} on $G$ is obtained by removing the edge $(l,w)$ and adding an edge $(\hat l,l)$ where $w$ is the nearest vertex (possibly $v$) to $l$. We let $r: = \mathrm{dist}_G(l,v)+1$ (the number of vertices in the chain from $l$ to $v$), and $\hat r:=\mathrm{dist}_G(\hat l,v)+1$ (see Figure \ref{flossdiag}). \\

\begin{figure}[H]
\caption{A graph $G$ followed by a flossing move applied to $G$. In this example we have $r=4$ and $\hat r=7$. The loop represents $G$ minus the path of vertices from $l$ to $\hat l$ that contains $v$.}

\label{flossdiag}
\[
\psset{unit=0.8cm}
\begin{pspicture}(4,0)(14.2,4.5)
\pscurve(7,2)(6.2,3)(7,4)(7.8,3)(7,2)
\qdisk(7,2){2.25pt}
\rput(7,1.5){$v$}
\psline(7,2)(5.5,1.55)
\psline(7,2)(10.6,1.4)
%\qdisk(5,1.4){2.25pt}
\qdisk(5.5,1.55){2.25pt}
\qdisk(6,1.7){2.25pt}
\qdisk(6.5,1.85){2.25pt}
\qdisk(10.6,1.4){2.25pt}
\rput(6,1.3){$w$}
\rput(5.5,1.15){$l$}
\qdisk(7.6,1.9){2.25pt}
\qdisk(8.2,1.8){2.25pt}
\qdisk(8.8,1.7){2.25pt}
\qdisk(9.4,1.6){2.25pt}
\qdisk(10,1.5){2.25pt}
\qdisk(10.6,1.4){2.25pt}
\rput(10.6,1){$\hat l$}
\pscurve[linewidth=.5pt](6.5,3.2)(7.2,3.6)(6.7,2.4)(7.45,2.8)
%\rput(6.8, 0.15){$v$}
%(6.6,0)(14.2,4.5)
\end{pspicture}
\begin{pspicture}(7,0)(14.2,4.5)
\pscurve(7,2)(6.2,3)(7,4)(7.8,3)(7,2)
\qdisk(7,2){2.25pt}
\rput(7,1.5){$v$}
\psline(7,2)(6,1.7)
\psline(7,2)(11.2,1.3)
%\qdisk(5,1.4){2.25pt}
%\qdisk(5.5,1.55){2.25pt}
\qdisk(6,1.7){2.25pt}
\qdisk(6.5,1.85){2.25pt}
\qdisk(10.6,1.4){2.25pt}
\rput(6,1.3){$w$}
%\rput(5.5,1.15){$l$}
\qdisk(7.6,1.9){2.25pt}
\qdisk(8.2,1.8){2.25pt}
\qdisk(8.8,1.7){2.25pt}
\qdisk(9.4,1.6){2.25pt}
\qdisk(10,1.5){2.25pt}
\qdisk(10.6,1.4){2.25pt}
\rput(10.6,1){$\hat l$}
\qdisk(11.2,1.3){2.25pt}
\rput(11.2,.9){$ l$}
\pscurve[linewidth=.5pt](6.5,3.2)(7.2,3.6)(6.7,2.4)(7.45,2.8)
%\rput(6.8, 0.15){$v$}
%(6.6,0)(14.2,4.5)
\end{pspicture}
\]
\end{figure}
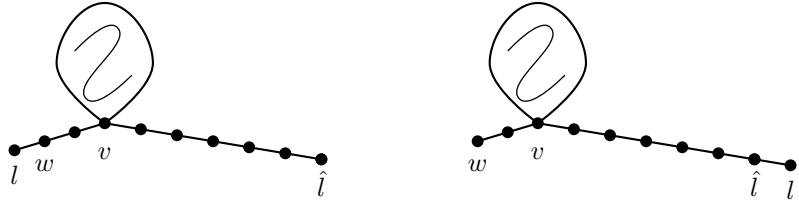

%Let $\delta(G)_i$ denote the number of pairs of vertices $v,l$ in $G$ such that there exists a leaf $\hat l$ such that $l,\hat l$ floss the vertex $v$ and $\mathrm{dist}_G(l,v) = i$. Then a flossing move either
%
%\begin{itemize}
%\item lowers $\delta(G)_i$ and $\delta(G)_j$ by one for some $i \le j$, and increases $\delta(G)_{i-1}$ and $\delta(G)_{j+1}$ by one (if $j = i$ then $\delta(G)_i$ is lowered by two).
%\item lowers $\delta(G)_1$ and $\delta(G)_j$ by one for some $j \ge 1$, and increases $\delta(G)_{j+1}$ by one.
%\item lowers $\delta(G)_1$ and $\delta(G)_j$ by one for some $j \ge 1$ (if $j=1$ then $\delta(G)_1$ lowers by two).
%\end{itemize}
%The first case happens when we apply the flossing move to a triple $(l, \hat l,v)$ where $\mathrm{dist}_G(l,v)=i$ for $i > 1$. The second case is when $\mathrm{dist}_G(l,v) =1$ and there exists a leaf $l'$ distinct from $l$ such that $l', \hat l$ floss the vertex $v$. The third case occurs when we apply the flossing move to a triple $(l, \hat l,v)$ such that $\mathrm{dist}_G(l,v) = 1$ and there is no leaf $l'$ distinct from $l$ such that $l',\hat l$ floss $v$. In each case we have let $j = \mathrm{dist}_G(\hat l,v)$.

\begin{thm}
Let $G$ be a connected graph, and let $G'$ be the resulting flossing move of $G$. Then $\gamma(\mathcal{B}(G')) \le \gamma(\mathcal{B}(G))$. \label{bigthmtwo}
\end{thm}

\begin{proof}
We suppose that $G$ has $n$ vertices, and we label $G$ by $l,\hat l,r,\hat r,v$ and $w,$ as in the definition of flossing move. We assume by induction that for any graph with $< n$ vertices, that a flossing move lowers the $\gamma$-polynomial. When $n <4$ no flossing move is possible so the result is vacuously true. $\mathcal{B}(G)$ is a flag building set on $[n]$, and the building set $\hat{\mathcal{B}}$ that is obtained from $\mathcal{B}(G)$ by removing all building set elements that contain $\{l,w\}$ apart from $[n]$ is also a flag building set on $[n]$. Hence by Theorem \ref{flagbuild}, $\mathcal{B}(G)$ can be obtained from $\hat{\mathcal{B}}$ by successively adding building set elements so that at each step the set is a flag building set. Similarly, $\mathcal{B}(G')$ can be obtained from $\hat{\mathcal{B}}$ by successively adding building set elements so that at each step the set is a flag building set. Similar to the arguments used in the proof of Theorem \ref{bigthm}, we construct an injection

$$\mathcal{B}(G') - \hat{\mathcal{B}} \rightarrow \mathcal{B}(G) - \hat{\mathcal{B}}$$
$$I' \mapsto I.$$ We then show that the increase in the $\gamma$-polynomial when adding the element in $\mathcal{B}(G') - \hat{\mathcal{B}}$ is less than or equal to the increase when adding the corresponding element in $\mathcal{B}(G) - \hat{\mathcal{B}}$ which proves the Theorem.\\

Let $I_1,I_2,...,I_k$ be the building set elements of $\mathcal{B}(G')-\hat{\mathcal{B}}$. Suppose for some $i \ne j$ that $I_j \subseteq I_i$. Then $j > i$, since $I_j \cap (I_i - \{l\}) \ne \emptyset$ and $I_j \cup (I_i - \{l\}) = I_i$ which implies that when $I_j$ is in the building set $I_i$ must be too.\\

Let $P$ be the set of vertices in the minimal path from $l$ to $\hat l$. Let $I'$ be an element that is added to $\hat{\mathcal{B}}$ to obtain $\mathcal{B}(G')$. There are three cases for $I'$ that we will consider.
\begin{itemize}
\item $|I'| \le \hat r,$
\item $|I'|\ge \hat r+1$, and $I'$ does not contain all of $G-P$,
\item $|I'| \ge \hat r+1$, and $I'$ contains all of $G-P$.
\end{itemize}
%See Figure \ref{diagrameight} for an example of when $|b'| \le \hat r$.\\
Suppose that $|I'| \le \hat r$, and let $I$ be the element of $\mathcal{B}(G')-\hat{\mathcal{B}}$ such that $|I \cap P| = r+ \hat r - |I'|,$ and $I$ contains all of $G-P$. In each case we let $\mathcal{B}_1$ (respectively $\mathcal{B}_2$) denote the building sets we have before adding $I$ (respectively $I'$). Then $\mathcal{B}_1|_{I} = \mathcal{B}_2/I' \cup \{\{l\}\}$, so that $\gamma(\mathcal{B}_1|_{I}) = \gamma(\mathcal{B}_2/I')$. Also, $\mathcal{B}_1/I \cup \{\{l\}\}= \mathcal{B}_2|_{I'}$, so that $\gamma(\mathcal{B}_1/I) = \gamma(\mathcal{B}_2|_{I'})$ (see Figure \ref{firstfloss}).\\

\begin{figure}[H]
\caption{The graph $G$ followed by $G'$. Keeping with the values of Figure \ref{flossdiag}, we have $|I'| =5$ and $|I \cap P| = 6$.}
\label{firstfloss}

\[
\psset{unit=0.8cm}
\begin{pspicture}(4.5,0)(14.2,4.5)
\pscurve(7,2)(6.2,3)(7,4)(7.8,3)(7,2)
\qdisk(7,2){2.25pt}
\rput(7,1.55){$v$}
\psline(7,2)(5.5,1.55)
\psline(7,2)(10.6,1.4)
%\qdisk(5,1.4){2.25pt}
\qdisk(5.5,1.55){2.25pt}
\qdisk(6,1.7){2.25pt}
\qdisk(6.5,1.85){2.25pt}
\qdisk(10.6,1.4){2.25pt}
\rput(6,1.3){$w$}
\rput(5.55,1.2){$l$}
\qdisk(7.6,1.9){2.25pt}
\qdisk(8.2,1.8){2.25pt}
\qdisk(8.8,1.7){2.25pt}
\qdisk(9.4,1.6){2.25pt}
\qdisk(10,1.5){2.25pt}
\qdisk(10.6,1.4){2.25pt}
\rput(10.6,1){$\hat l$}
\rput(8.6,2.7){$I$}
\pscurve[linewidth=.5pt](6.5,3.2)(7.2,3.6)(6.7,2.4)(7.45,2.8)
\pscurve[linestyle=dashed](7.8,2.3)(8.5,3)(7,4.7)(5.5,3)(6,2)(5.1,1.6)(5.1,1.3)(5.6,.9)(7.3,1.4)(8.2,1.3)(8.4,2.1)(7.8,2.3)
\psline(5.7,1.75)(5.7,1.48)
\psline(5.8,1.78)(5.8,1.5)
%(6.6,0)(14.2,4.5)
\end{pspicture}
\begin{pspicture}(7,0)(14.2,4.5)
\pscurve(7,2)(6.2,3)(7,4)(7.8,3)(7,2)
\qdisk(7,2){2.25pt}
\rput(7,1.5){$v$}
\psline(7,2)(6,1.7)
\psline(7,2)(11.2,1.3)
%\qdisk(5,1.4){2.25pt}
%\qdisk(5.5,1.55){2.25pt}
\qdisk(6,1.7){2.25pt}
\qdisk(6.5,1.85){2.25pt}
\qdisk(10.6,1.4){2.25pt}
\rput(6,1.3){$w$}
%\rput(5.5,1.15){$l$}
\qdisk(7.6,1.9){2.25pt}
\qdisk(8.2,1.8){2.25pt}
\qdisk(8.8,1.7){2.25pt}
\qdisk(9.4,1.6){2.25pt}
\qdisk(10,1.5){2.25pt}
\qdisk(10.6,1.4){2.25pt}
\rput(10.6,1){$\hat l$}
\qdisk(11.2,1.3){2.25pt}
\rput(11.2,.9){$l$}
\pscurve[linewidth=.5pt](6.5,3.2)(7.2,3.6)(6.7,2.4)(7.45,2.8)
\pscurve[linestyle=dashed](8.5,1.75)(8.6,2.2)(10,2.2)(11.6,1.75)(11.6,.8)(10,.7)(8.6,1.2)(8.5,1.75)
\rput(9.5,2.8){$I'$}
\psline(10.85,1.52)(10.85,1.24)
\psline(10.955,1.48)(10.955,1.2)
%\pscurve[linestyle=dashed]()()()
%\rput(6.8, 0.15){$v$}
%(6.6,0)(14.2,4.5)
\end{pspicture}
\]
\end{figure}
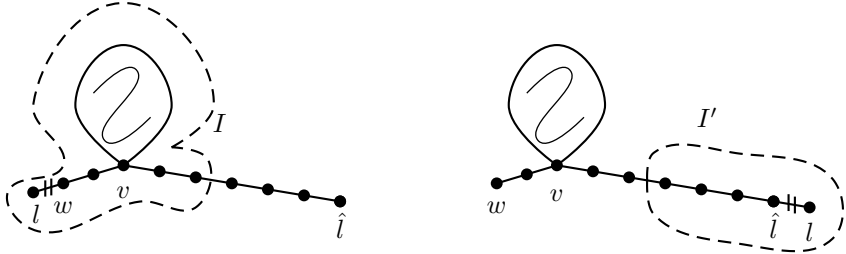
%$B_2$
%
%\[
%\psset{unit=0.8cm}
%\begin{pspicture}(6.6,0)(14.2,4.5)
%\pscurve(7,2)(6.2,3)(7,4)(7.8,3)(7,2)
%\qdisk(7,2){2.25pt}
%\rput(7,1.5){$v$}
%\psline(7,2)(6,1.7)
%\psline(7,2)(11.2,1.3)
%%\qdisk(5,1.4){2.25pt}
%%\qdisk(5.5,1.55){2.25pt}
%\qdisk(6,1.7){2.25pt}
%\qdisk(6.5,1.85){2.25pt}
%\qdisk(10.6,1.4){2.25pt}
%\rput(6,1.3){$w$}
%%\rput(5.5,1.15){$l$}
%\qdisk(7.6,1.9){2.25pt}
%\qdisk(8.2,1.8){2.25pt}
%\qdisk(8.8,1.7){2.25pt}
%\qdisk(9.4,1.6){2.25pt}
%\qdisk(10,1.5){2.25pt}
%\qdisk(10.6,1.4){2.25pt}
%\rput(10.6,1){$\hat l$}
%\qdisk(11.2,1.3){2.25pt}
%\rput(11.2,.9){$l$}
%\pscurve[linewidth=.5pt](6.5,3.2)(7.2,3.6)(6.7,2.4)(7.45,2.8)
%\pscurve[linestyle=dashed](8.5,1.75)(8.6,2.2)(10,2.2)(11.6,1.75)(11.6,.8)(10,.7)(8.6,1.2)(8.5,1.75)
%\rput(9.5,2.8){$b'$}
%\psline(10.85,1.52)(10.85,1.24)
%\psline(10.955,1.48)(10.955,1.2)
%%\pscurve[linestyle=dashed]()()()
%%\rput(6.8, 0.15){$v$}
%\end{pspicture}
%\]

Suppose that $|I'| \ge \hat r +1$, and suppose that $I'$ does not contain all of $G-P$. Let $I$ be the element of $\mathcal{B}(G)-\hat{\mathcal{B}}$ such that $|I \cap P| = |I' \cap P|$, and $I \cap (G-P) = I' \cap (G-P)$. Then we have that $\mathcal{B}_1/I \cong \mathcal{B}_2/I'$, and $\mathcal{B}_1|_{I} = \mathcal{B}(G_1) \cup \{\{l\}\}$, and $\mathcal{B}_2|_{I'} = \mathcal{B}(G_2) \cup \{\{l\}\}$ where $G_2$ is a graph obtained from a graph $G_1$ by a flossing move (or if $\mathrm{dist}_G(l,v) =1$, $G_2 = G_1$). By induction on the number of vertices of the graphs involved we have that $\gamma(\mathcal{B}(G_2)) \le \gamma(\mathcal{B}(G_1))$ so that $\gamma(\mathcal{B}_2|_{I'}) \le \gamma(\mathcal{B}_1|_{I})$ (see Figure \ref{flosstwo}).\\

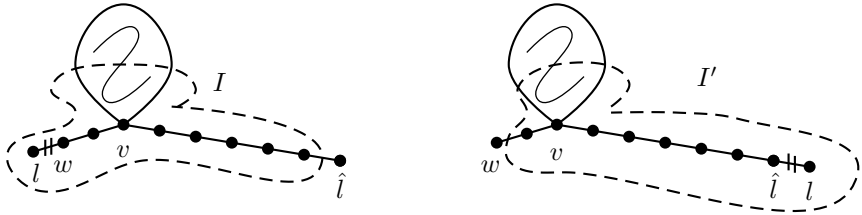
\begin{figure}[H]

\caption{The graph $\mathcal{B}_1$ followed by $\mathcal{B}_2$. We have $|I'| \ge 7$.}
\label{flosstwo}
\[
\psset{unit=0.8cm}
\begin{pspicture}(4.5,0)(14.2,4.5)
\pscurve(7,2)(6.2,3)(7,4)(7.8,3)(7,2)
\qdisk(7,2){2.25pt}
\rput(7,1.55){$v$}
\psline(7,2)(5.5,1.55)
\psline(7,2)(10.6,1.4)
%\qdisk(5,1.4){2.25pt}
\qdisk(5.5,1.55){2.25pt}
\qdisk(6,1.7){2.25pt}
\qdisk(6.5,1.85){2.25pt}
\qdisk(10.6,1.4){2.25pt}
\rput(6,1.3){$w$}
\rput(5.55,1.2){$l$}
\qdisk(7.6,1.9){2.25pt}
\qdisk(8.2,1.8){2.25pt}
\qdisk(8.8,1.7){2.25pt}
\qdisk(9.4,1.6){2.25pt}
\qdisk(10,1.5){2.25pt}
\qdisk(10.6,1.4){2.25pt}
\rput(10.6,1){$\hat l$}
\rput(8.6,2.7){$I$}
\pscurve[linewidth=.5pt](6.5,3.2)(7.2,3.6)(6.7,2.4)(7.45,2.8)
\pscurve[linestyle=dashed](7.9,2.3)(8.2,2.6)(7,3)(5.8,2.6)(6.2,2.2)(5.1,1.6)(5.1,1.3)(5.6,.9)(7.3,1.4)
(10.3,1.4)(7.8,2.3)
\psline(5.7,1.75)(5.7,1.48)
\psline(5.8,1.78)(5.8,1.5)
%\rput(6.8, 0.15){$v$}
%(6.6,0)(14.2,4.5)
\end{pspicture}
\begin{pspicture}(7,0)(14.2,4.5)
\pscurve(7,2)(6.2,3)(7,4)(7.8,3)(7,2)
\qdisk(7,2){2.25pt}
\rput(7,1.5){$v$}
\psline(7,2)(6,1.7)
\psline(7,2)(11.2,1.3)
%\qdisk(5,1.4){2.25pt}
%\qdisk(5.5,1.55){2.25pt}
\qdisk(6,1.7){2.25pt}
\qdisk(6.5,1.85){2.25pt}
\qdisk(10.6,1.4){2.25pt}
\rput(5.9,1.3){$w$}
%\rput(5.5,1.15){$l$}
\qdisk(7.6,1.9){2.25pt}
\qdisk(8.2,1.8){2.25pt}
\qdisk(8.8,1.7){2.25pt}
\qdisk(9.4,1.6){2.25pt}
\qdisk(10,1.5){2.25pt}
\qdisk(10.6,1.4){2.25pt}
\rput(10.6,1){$\hat l$}
\qdisk(11.2,1.3){2.25pt}
\rput(11.2,.9){$ l$}
\pscurve[linewidth=.5pt](6.5,3.2)(7.2,3.6)(6.7,2.4)(7.45,2.8)
%\pscurve[linestyle=dashed](8.5,1.75)(8.6,2.2)(10,2.2)(11.6,1.75)(11.6,.8)(10,.7)(8.6,1.2)(8.5,1.75)
\pscurve[linestyle=dashed](7.9,2.3)(8.2,2.7)(7,3)(6.35,2.6)(6.5,2.2)(6.2,1.8)(6.2,1.4)(7.3,1.2)
(12,1.2)(10,2.1)(9,2.2)(7.9,2.2)
\rput(9.5,2.8){$I'$}
\psline(10.85,1.52)(10.85,1.24)
\psline(10.955,1.48)(10.955,1.2)
%\pscurve[linestyle=dashed]()()()
%\rput(6.8, 0.15){$v$}
%(6.6,0)(14.2,4.5)
\end{pspicture}
\]
\end{figure}

Suppose that $|I'| \ge \hat r +1$ and $I'$ contains all of $G-P$. Let $I$ be the element of $\mathcal{B}(G) - \hat{\mathcal{B}}$ such that $|I| = r+ \hat r - |I' \cap P|$. Then $\mathcal{B}_1/I \cup \{\{l\}\} = \mathcal{B}_2|_{I'}$, and $\mathcal{B}_1|_{I} = \mathcal{B}_2/I' \cup \{\{l\}\}$. Hence $\gamma(\mathcal{B}_1/I) = \gamma(\mathcal{B}_2|_{I'})$ and $\gamma(\mathcal{B}_1|_{I}) = \gamma(\mathcal{B}_2/I')$ (see Figure \ref{flossthree}).\\

\begin{figure}[H]

\caption{The graph $\mathcal{B}_1$ followed by $\mathcal{B}_2$. We have $|I| =2$ and $|I' \cap P| = 9$.}
\label{flossthree}
\[
\psset{unit=0.8cm}
\begin{pspicture}(4.5,0)(14.2,4.5)
\pscurve(7,2)(6.2,3)(7,4)(7.8,3)(7,2)
\qdisk(7,2){2.25pt}
\rput(7,1.55){$v$}
\psline(7,2)(5.5,1.55)
\psline(7,2)(10.6,1.4)
%\qdisk(5,1.4){2.25pt}
\qdisk(5.5,1.55){2.25pt}
\qdisk(6,1.7){2.25pt}
\qdisk(6.5,1.85){2.25pt}
\qdisk(10.6,1.4){2.25pt}
\rput(6,1.2){$w$}
\rput(5.55,1.1){$l$}
\qdisk(7.6,1.9){2.25pt}
\qdisk(8.2,1.8){2.25pt}
\qdisk(8.8,1.7){2.25pt}
\qdisk(9.4,1.6){2.25pt}
\qdisk(10,1.5){2.25pt}
\qdisk(10.6,1.4){2.25pt}
\rput(10.6,1){$\hat l$}
\rput(5,2){$I$}
\pscurve[linewidth=.5pt](6.5,3.2)(7.2,3.6)(6.7,2.4)(7.45,2.8)
%\pscurve[linestyle=dashed](7.9,2.3)(8.2,2.6)(7,3)(5.8,2.6)(6.2,2.2)(5.1,1.6)(5.1,1.3)(5.6,.9)(7.3,1.4)
%(10.3,1.4)(7.8,2.3)
\pscurve(5.2,1.55)(5.75,1.8)(6.25,1.8)(5.75,1.4)(5.2,1.4)(5.2,1.55)
\psline(5.7,1.75)(5.7,1.48)
\psline(5.8,1.78)(5.8,1.5)
%\rput(6.8, 0.15){$v$}
%(6.6,0)(14.2,4.5)
\end{pspicture}
\begin{pspicture}(7,0)(14.2,4.5)
\pscurve(7,2)(6.2,3)(7,4)(7.8,3)(7,2)
\qdisk(7,2){2.25pt}
\rput(7,1.5){$v$}
\psline(7,2)(6,1.7)
\psline(7,2)(11.2,1.3)
%\qdisk(5,1.4){2.25pt}
%\qdisk(5.5,1.55){2.25pt}
\qdisk(6,1.7){2.25pt}
\qdisk(6.5,1.85){2.25pt}
\qdisk(10.6,1.4){2.25pt}
\rput(5.9,1.3){$w$}
%\rput(5.5,1.15){$l$}
\qdisk(7.6,1.9){2.25pt}
\qdisk(8.2,1.8){2.25pt}
\qdisk(8.8,1.7){2.25pt}
\qdisk(9.4,1.6){2.25pt}
\qdisk(10,1.5){2.25pt}
\qdisk(10.6,1.4){2.25pt}
\rput(10.6,1){$\hat l$}
\qdisk(11.2,1.3){2.25pt}
\rput(11.2,.9){$ l$}
\pscurve[linewidth=.5pt](6.5,3.2)(7.2,3.6)(6.7,2.4)(7.45,2.8)
%\pscurve[linestyle=dashed](8.5,1.75)(8.6,2.2)(10,2.2)(11.6,1.75)(11.6,.8)(10,.7)(8.6,1.2)(8.5,1.75)
\pscurve[linestyle=dashed](7.9,2.3)(8.2,2.7)(7,4.2)(6.,2.6)(6.4,2.2)(6.2,1.8)(6.2,1.4)(7.3,1.2)
(12,1.2)(10,2.1)(9,2.2)(7.9,2.2)
\rput(9.5,2.8){$I'$}
\psline(10.85,1.52)(10.85,1.24)
\psline(10.955,1.48)(10.955,1.2)
%\pscurve[linestyle=dashed]()()()
%\rput(6.8, 0.15){$v$}
%(6.6,0)(14.2,4.5)
\end{pspicture}
\]
\end{figure}
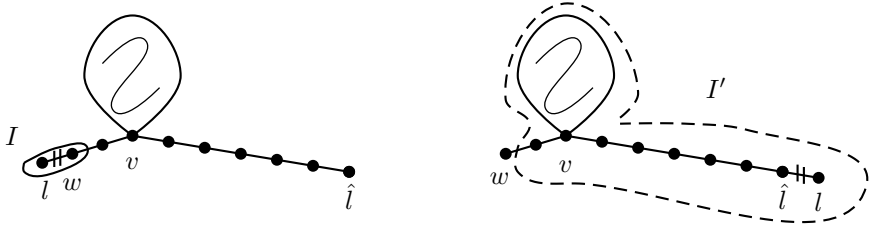

Note that no element $I \in \mathcal{B}(G)- \hat{\mathcal{B}}$ is used more than once, since in the first case we have that $|I| \ge r$ and $I$ contains all of $G-P$. In the second case we have that $|I| \ge \hat r + 1 >r$ and $I$ does not contain all of $G-P$. In the third case we have that $|I| = r + \hat r - |I' \cap P| \le r + \hat r - (\hat r +1) = r-1$. \\

By Lemma \ref{vollem} the change in the $\gamma$-polynomial when adding $I'$ is given by $$\gamma(\mathcal{B}_2 \cup \{I'\}) - \gamma(\mathcal{B}_2) = t\gamma(\mathcal{B}_2/I')\gamma(\mathcal{B}_2|_{I'}),$$ and when adding $I$ it is given by $$\gamma(\mathcal{B}_1 \cup \{I\}) - \gamma(\mathcal{B}_1) = t\gamma(\mathcal{B}_1/I)\gamma(\mathcal{B}_1|_{I}).$$ Since for every element $I'$ that is added to $\hat{\mathcal{B}}$ to obtain $\mathcal{B}(G')$, there is an element $I$ that is added to $\hat{\mathcal{B}}$ to obtain $\mathcal{B}(G)$ such that $\gamma(\mathcal{B}_2/I')\gamma(\mathcal{B}_2|_{I'}) \le \gamma(\mathcal{B}_1/I)\gamma(\mathcal{B}_1|_{I})$ we have that $\gamma(\mathcal{B}(G')) \le  \gamma(\mathcal{B}(G))$.

\end{proof}

It is exactly when $\mathrm{dist}_G(l,v)=1$ that a flossing move is a kind of tree shift. This is exactly when a flossing move reduces the number of leaves. If we partition the set $S$ of all tree graphs with $n$ vertices by their number of leaves, then tree shifts send graphs between the parts, whilst flossing moves such that $\mathrm{dist}_T(l,v) \ne 1$ give relations between graphs with the same number of leaves. This is illustrated in the following example for tree graphs with seven vertices.

\begin{example}

Arrows are drawn between pairs of graphs with the same number of leaves when one (at the head) can be obtained from the other (at the tail) by a flossing move. Arrows are drawn from a graph with $i+1$ leaves to one with $i$ leaves when the graph at the head can be obtained from the graph at the tail by a tree shift.

\begin{figure}[H]

\[
\psset{unit=0.5cm}
\begin{pspicture}(-1,-1)(1,1)
\rput(-8.5,0){$6 ~\hbox{leaves}$}
\psline(.5,0.866)(-.5,-0.866)
\psline(1,0)(-1,0)
\psline(-.5, 0.866)(.5,-0.866)
\qdisk(.5,0.866){2.25pt}\qdisk(-.5,-0.866){2.25pt}
\qdisk(1,0){2.25pt}\qdisk(-1,0){2.25pt}
\qdisk(-.5, 0.866){2.25pt}\qdisk(.5,-0.866){2.25pt}
\qdisk(0,0){2.25pt}
\psline{->}(0,-1.2)(-1.2,-3)
\psline{->}(0,-1.2)(1.2,-3)
\end{pspicture}
\]
\vspace{.5cm}
\[
\psset{unit=0.5cm}
\begin{pspicture}(-2.5,-1)(2.5,1)
\rput(-8.5,0){$5 ~\hbox{leaves}$}
\psline(-2.1,0)(-0.5,0)
\psline(-2.1,0)(-2.1,.8)
\psline(-2.1,0)(-2.1,-.8)
\psline(-2.1,0)(-2.7,.5)
\psline(-2.1,0)(-2.7,-.5)
\qdisk(-2.1,0){2.25pt}\qdisk(-.5,0){2.25pt}
\qdisk(-2.1,.8){2.25pt}\qdisk(-2.1,-.8){2.25pt}
\qdisk(-2.7, 0.5){2.25pt}\qdisk(-2.7,-0.5){2.25pt}
\qdisk(-1.3,0){2.25pt}

\psline(0.5,0)(2.1,0)
\qdisk(0.5,0){2.25pt}\qdisk(2.1,0){2.25pt}\qdisk(1.3,0){2.25pt}
\psline(1.3,.8)(1.3,-.8)\qdisk(1.3,.8){2.25pt}\qdisk(1.3,-.8){2.25pt}
\psline(2.1,0)(2.7,.5)\qdisk(2.7,.5){2.25pt}
\psline(2.1,0)(2.7,-.5)\qdisk(2.7,-.5){2.25pt}

\psline{->}(-2,-1.2)(-2,-3.5)
\psline{->}(-2,-1.2)(-4.5,-3.5)
\psline{->}(-2,-1.2)(1.5,-3.5)
\psline{->}(-2,-1.2)(4.2,-3.5)

\psline{->}(2,-1.2)(-1.7,-3.5)
%\psline{->}(2,-1.2)(-4.2,-3.5)
\psline{->}(2,-1.2)(1.8,-3.5)
\psline{->}(2,-1.2)(4.5,-3.5)
\end{pspicture}
\]

\vspace{.5cm}
\[
\psset{unit=0.5cm}
\begin{pspicture}(-6.5,-1)(6.5,1)
\rput(-8.5,0){$4~\hbox{leaves}$}
\psline(-5.85,-1)(-5.85,1)
\psline(-6.5,-.33)(-4.5,-.33)
\qdisk(-5.85,-1){2.25pt}\qdisk(-5.85,1){2.25pt}
\qdisk(-6.5,-.33){2.25pt}\qdisk(-4.5,-.33){2.25pt}
\qdisk(-5.85,-.33){2.25pt}\qdisk(-5.85,.33){2.25pt}\qdisk(-5.18,-.33){2.25pt}

\pscurve{->}(-5.25,.8)(-4.25,1.3)(-3.25,.8)

\psline(-3.5,0)(-.7,0)
\psline(-2.8,-.7)(-2.8,.7)
\qdisk(-3.5,0){2.25pt}
\qdisk(-2.8,-.7){2.25pt}\qdisk(-2.8,.7){2.25pt}
\qdisk(-2.8,0){2.25pt}\qdisk(-2.1,0){2.25pt}\qdisk(-1.4,0){2.25pt}\qdisk(-.7,0){2.25pt}

\psline(.5,1)(1.3,-.33)
\psline(.9,-1)(1.3,-.33)
\psline(1.3,-.33)(2.1,-.33)
\psline(2.1,-.33)(2.5,.33)
\psline(2.1,-.33)(2.5,-1)
\qdisk(.5,1){2.25pt}\qdisk(1.3,-.33){2.25pt}\qdisk(.9,-1){2.25pt}\qdisk(1.3,-.33){2.25pt}
\qdisk(1.3,-.33){2.25pt}\qdisk(2.1,-.33){2.25pt}\qdisk(2.5,.33){2.25pt}\qdisk(2.1,-.33){2.25pt}\qdisk(2.5,-1){2.25pt}
\qdisk(.9,.33){2.25pt}

\psline(4.4,0)(6,0)
\psline(4.4,0)(4,.66)
\psline(4.4,0)(4,-.66)
\psline(6,0)(6.4,.66)
\psline(6,0)(6.4,-.66)
\qdisk(4.4,0){2.25pt}\qdisk(6,0){2.25pt}\qdisk(4,.66){2.25pt}\qdisk(4,-.66){2.25pt}
\qdisk(6.4,.66){2.25pt}\qdisk(6.4,-.66){2.25pt}
\qdisk(5.2,0){2.25pt}

%\psline[linestyle=dashed]{->}(-2,-1)(-3.5,-4)
%\psline[linestyle=dashed]{->}(5.4,-1)(1.4,-4.3)

\psline{->}(-4.5,-1)(-3.5,-4.5)
\psline{->}(-4.5,-1)(-.1,-4.5)
\psline{->}(-4.5,-1)(3.7,-4.5)

\psline{->}(-2,-1)(.2,-4.5)
\psline{->}(-2,-1)(4,-4.5)

\psline{->}(1.7,-1)(.5,-4.5)
\psline{->}(1.7,-1)(4.3,-4.5)

\psline{->}(5.4,-1)(4.6,-4.5)
\end{pspicture}
\]
\vspace{.5cm}

\[
\psset{unit=0.5cm}
\begin{pspicture}(-4,-1)(4,1)
\rput(-8.5,0){$3 ~\hbox{leaves}$}
\psline(-3.5,0)(-2.3,0)
\psline(-3.5,0)(-4.3,1)
\psline(-3.5,0)(-4.3,-1)
\qdisk(-3.5,0){2.25pt}\qdisk(-2.3,0){2.25pt}
\qdisk(-4.3,-1){2.25pt}\qdisk(-4.3,1){2.25pt}
\qdisk(-3.9,.5){2.25pt}\qdisk(-3.9,-.5){2.25pt}
\qdisk(-2.9,0){2.25pt}

\pscurve{->}(-3,1.2)(-2,1.7)(-1,1.2)

\psline(-.5,0)(1.3,0)
\psline(-.5,0)(-1.3,1)
\psline(-.5,0)(-.9,-.5)
\qdisk(-.5,0){2.25pt}\qdisk(1.3,0){2.25pt}\qdisk(-1.3,1){2.25pt}\qdisk(-.9,-.5){2.25pt}
\qdisk(-.9,.5){2.25pt}\qdisk(.1,0){2.25pt}\qdisk(.7,0){2.25pt}\qdisk(1.3,0){2.25pt}

\pscurve{->}(.8,.8)(1.8,1.3)(2.8,.8)

\psline(2.7,0)(2.3,.5)\psline(2.7,0)(2.3,-.5)
\qdisk(2.3,.5){2.25pt}\qdisk(2.3,-.5){2.25pt}
\psline(2.7,0)(5.1,0)
\qdisk(2.7,0){2.25pt}\qdisk(3.3,0){2.25pt}\qdisk(3.9,0){2.25pt}\qdisk(4.5,0){2.25pt}\qdisk(5.1,0){2.25pt}

\psline{->}(-2.5,-1)(.2,-4.2)
\psline{->}(.5,-1)(.5,-4.2)
\psline{->}(3.5,-1)(.8,-4.2)
\end{pspicture}
\]

\vspace{.5cm}
\[
\psset{unit=0.5cm}
\begin{pspicture}(-1.8,-1)(1.8,1)
\rput(-8.5,0){$2~~leaves$}
\psline(-1.8,0)(1.8,0)
\qdisk(-1.8,0){2.25pt}\qdisk(-1.2,0){2.25pt}\qdisk(-.6,0){2.25pt}
\qdisk(0,0){2.25pt}\qdisk(.6,0){2.25pt}\qdisk(1.2,0){2.25pt}\qdisk(1.8,0){2.25pt}
\end{pspicture}
\]
\caption{Tree graphs with 7 vertices and their tree shift and flossing move relations.}
\end{figure}
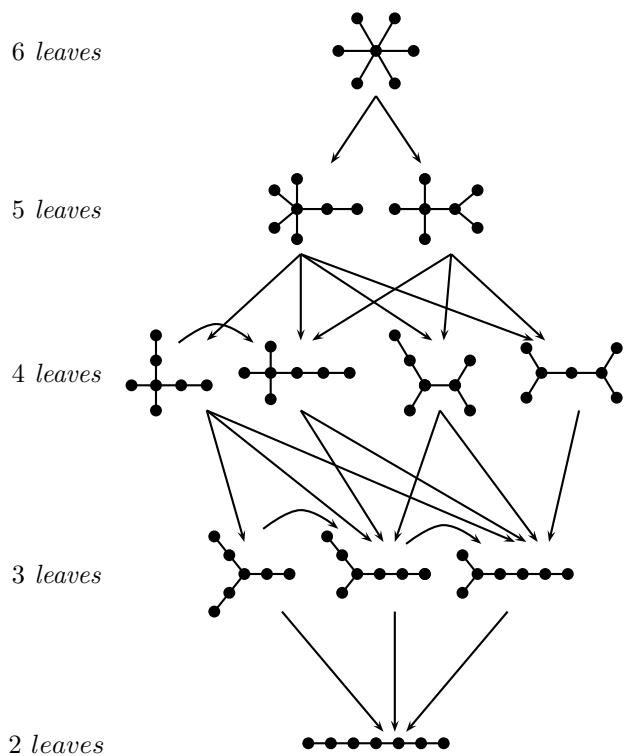
\end{example}

\end{section}

It is suggested in \cite{prw} that a move on a tree graph that increases the \emph{Wiener index} \cite{wie} might approximately lower the $\gamma$-polynomial, although the only moves that we have found that increase the Wiener index and lower the $\gamma$-polynomial are tree shifts and flossing moves.

%\begin{proof}[Proof of Theorem \ref{bigthm}]
%\end{proof}
\end{document}